\documentclass{amsart}
\usepackage{amsfonts,amssymb,amsmath,amscd,latexsym,amsthm,amsbsy,dsfont,hyperref}
\usepackage{cleveref,tikz}
\usepackage{pgfplots}
\pgfplotsset{compat=1.17} 
\tikzset{>=latex}

\usepackage{graphicx,pgf,caption,subcaption}
\usepackage{enumerate}

\hypersetup{
    pdftoolbar=true,
    pdfmenubar=true,
    pdffitwindow=false,
    pdfstartview={FitH},
    pdftitle={},
    pdfauthor={},
    pdfsubject={},
    pdfkeywords={}
    pdfnewwindow=true,
    colorlinks=true, 
    linkcolor=blue,
    citecolor=blue,
    urlcolor=black,
}

\title[Ricci flow does not preserve positive sectional curvature]{Ricci flow does not preserve positive sectional curvature in dimension four}

\author[R. G. Bettiol]{Renato G. Bettiol}
\address{City University of New York (Lehman College), Department of Mathematics, \newline \indent 250 Bedford Park Blvd W, Bronx, NY, 10468, USA}
\email{r.bettiol@lehman.cuny.edu}

\author[A. M. Krishnan]{Anusha M. Krishnan}
\address{Universit\"at M\"unster, Mathematisches Institut,\newline\indent  Einsteinstrasse 62, 48149 M\"unster, Germany}
\email{anusha.krishnan@uni-muenster.de}

\thanks{The first-named author is supported by the National Science Foundation award DMS-1904342. The second-named author is funded by the Deutsche Forschungsgemeinschaft (DFG, German Research Foundation) under Germany's Excellence Strategy EXC 2044 -- 390685587, Mathematics Münster: Dynamics--Geometry--Structure.}

\date{\today}

\newcommand{\dd}{\mathrm{d}}

\newcommand{\R}{\mathds{R}}
\newcommand{\C}{\mathds{C}}
\newcommand{\Z}{\mathds{Z}}

\newcommand{\B}{\mathcal{B}}
\newcommand{\g}{\mathfrak{g}}

\renewcommand{\k}{\mathfrak{k}}
\newcommand{\h}{\mathfrak{h}}
\newcommand{\m}{\mathfrak{m}}
\newcommand{\p}{\mathfrak{p}}
\newcommand{\n}{\mathfrak{n}}
\newcommand{\G}{\mathsf{G}}
\newcommand{\K}{\mathsf{K}}
\renewcommand{\H}{\mathsf{H}}
\renewcommand{\S}{\mathsf{S}}
\newcommand{\T}{\mathsf{T}}

\newcommand{\SU}{\mathsf{SU}}
\newcommand{\SO}{\mathsf{SO}}
\newcommand{\OO}{\mathsf{O}}
\newcommand{\gm}{\mathrm g}
\newcommand{\gGZ}{\mathrm g_{\mathrm{GZ}}}
\newcommand{\RR}{R}
\newcommand{\Ad}{\operatorname{Ad}}
\newcommand{\diag}{\operatorname{diag}}
\newcommand{\Sym}{\operatorname{Sym}}
\newcommand{\Ric}{\operatorname{Ric}}

\newtheorem{theorem}{Theorem}
\newtheorem{lemma}[theorem]{Lemma}
\newtheorem{claim}[theorem]{Claim}

\newtheorem{proposition}[theorem]{Proposition}
\newtheorem{mainthm}{\sc Theorem}

\newtheorem*{mainthm*}{\sc Theorem}
\newtheorem*{notation*}{Notation}

\theoremstyle{remark}
\newtheorem{remark}[theorem]{Remark}

\allowdisplaybreaks
\numberwithin{equation}{section}
\numberwithin{theorem}{section}

\begin{document}
\begin{abstract}
We find examples of cohomogeneity one metrics on $S^4$ and $\C P^2$ with positive sectional curvature that lose this property when evolved via Ricci flow. These metrics are arbitrarily small perturbations of Grove--Ziller metrics with flat planes that become instantly negatively curved under Ricci flow.
\end{abstract}
\maketitle


\section{Introduction}

The Ricci flow $\tfrac{\partial}{\partial t}\gm(t)=-2\Ric_{\gm(t)}$ of Riemannian metrics on a smooth manifold is an evolution equation that continues to drive a wide range of breakthroughs in Geometric Analysis, see e.g.~\cite{bamler-survey} for a survey. One of the keys to using Ricci flow is to control how the curvature of $\gm(t)$ evolves; in particular, which curvature conditions of the original metric $\gm(0)$ are preserved. 
Our main result establishes that, in dimension $n=4$, positive sectional curvature ($\sec>0$)  is \emph{not} among them:

\begin{mainthm}\label{mainthmA}
There exist smooth Riemannian metrics with $\sec>0$ on $S^4$ and $\C P^2$ that evolve under the Ricci flow to metrics with sectional curvatures of mixed sign.
\end{mainthm}

In contrast, $\sec>0$ is preserved on closed manifolds of dimension $n\leq 3$, by the seminal work of Hamilton~\cite{hamilton-original}. Moreover, it was previously known~\cite{maximo2} that $\Ric>0$ is not preserved in dimension $n=4$, even among K\"ahler metrics, but these examples do not have $\sec>0$. Although \Cref{mainthmA} does not readily extend to all $n>4$, there are examples of homogeneous metrics on flag manifolds of dimensions $6$, $12$, and $24$ with $\sec>0$ that lose that property when evolved via Ricci flow, see \cite{bw-gafa-2007,cheung-wallach,abiev-nikonorov}.
A state-of-the-art discussion of Ricci flow invariant curvature conditions can be found in \cite{bcrw}, see also \Cref{rem:max-princ}.

\Cref{mainthmA} builds on our earlier result~\cite{bettiol-krishnan1} that certain metrics with $\sec\geq0$, introduced by Grove and Ziller~\cite{grove-ziller-annals} in a much broader context (see \Cref{subsec:GZ-metrics}), immediately acquire negatively curved planes on $S^4$ and $\C P^2$, when evolved under Ricci flow. In light of the appropriate continuous dependence of Ricci flow on its initial data \cite{BGI20}, the metrics in \Cref{mainthmA} are obtained by means~of:

\begin{mainthm}\label{mainthmB}
Every Grove--Ziller metric on $S^4$ or $\C P^2$ is the limit (in $C^\infty$-topology) of cohomogeneity one metrics with $\sec>0$.
\end{mainthm}

In full generality, the problem of perturbing $\sec\geq0$ to $\sec>0$ is notoriously difficult, see e.g.~\cite[Prob.~2]{wilking-survey}. Aside from clearly being unobstructed on $S^4$ and $\C P^2$, the deformation problem is facilitated here by the presence of natural directions for perturbation, given by the round metric and the Fubini--Study metric, respectively. Indeed, we deform $\sec\geq0$ into $\sec>0$ in \Cref{mainthmB} by linearly interpolating lengths of Killing vector fields for the $\SO(3)$-action which is isometric for both the Grove--Ziller metric $\gm_0$ and the standard metric $\gm_1$ on these spaces. The resulting $\SO(3)$-invariant metrics $\gm_s$, $s\in [0,1]$, are smooth and have $\sec>0$ for all sufficiently small $s>0$. For a lower-dimensional illustration, consider the $\T^2$-action on $S^3\subset \C^2$ via $(e^{i\theta_1},e^{i\theta_2})\cdot(z,w)=\big(e^{i\theta_1}z,e^{i\theta_2}w\big)$, and invariant metrics
\begin{equation*}
\gm=\dd r^2+ \varphi(r)^2\,\dd \theta_1^2+\xi(r)^2\,\dd \theta_2^2, \quad 0<r<\tfrac\pi2,
\end{equation*}
written along the geodesic segment $\gamma(r)=(\sin r,\cos r)$. 
The functions $\varphi$ and $\xi$ encode the $\gm$-lengths of the Killing fields $\frac{\partial}{\partial \theta_1}$ and $\frac{\partial}{\partial \theta_2}$ respectively, and must satisfy certain smoothness conditions at the endpoints $r=0$ and $r=\frac\pi2$. 
The unit round metric $\gm_1$ is given by setting $\varphi$ and $\xi$ to be $\varphi_1(r)=\sin r$ and $\xi_1(r)=\cos r$, while a Grove--Ziller metric $\gm_0$ corresponds to concave monotone functions $\varphi_0$ and $\xi_0$ that plateau at a constant value $b>0$ for at least half of $\left[0,\frac\pi2\right]$.
The curvature operator of $\gm$ is easily seen to be diagonal, with eigenvalues $-\varphi''/\varphi$, $-\xi''/\xi$, and $-\varphi'\xi'/\varphi\xi$, see e.g.~\cite[Sec.~4.2.4]{petersen-book-3}, so it has $\sec\geq0$ if and only if $\varphi$ and $\xi$ are concave and monotone, and $\sec>0$ if and only if they are \emph{strictly} concave and monotone. Thus, 
\begin{equation*}
\varphi_s=(1-s)\,\varphi_0+s\,\varphi_1 \quad \text{ and }\quad \xi_s=(1-s)\,\xi_0+s\,\xi_1
\end{equation*}
give rise to metrics $\gm_s$  deforming $\gm_0$ to have $\sec>0$ for $s>0$. It turns out that a similar approach works for proving \Cref{mainthmB}, with the addition of a third (nowhere vanishing) function $\psi$, to deal with $\SO(3)$-invariant metrics on $4$-manifolds. The biggest challenge is verifying that these metrics have $\sec>0$, since that is no longer equivalent to positive-definiteness of the curvature operator if $n\geq4$. To overcome this difficulty, we use a much simpler algebraic characterization of $\sec>0$ in dimension $n=4$, given by the Finsler--Thorpe trick (\Cref{prop:FTtrick}).

Motivated by the above, it is natural to ask whether the set of cohomogeneity one metrics with $\sec\geq0$ on a given closed manifold coincides with the closure (say, in $C^2$-topology) of the set of such metrics with $\sec>0$, if the latter is nonempty.
In contrast to \Cref{mainthmB}, there is some evidence to suggest that Grove--Ziller metrics on certain $7$-manifolds cannot be perturbed to have $\sec>0$, see~\cite[Sec.~4]{ziller-coh1survey}.

This paper is organized as follows. Background material on cohomogeneity one manifolds and the Finsler--Thorpe trick in dimension $4$ is presented in \Cref{sec:prelim}. The smoothness conditions and curvature operator of $\SO(3)$-invariant metrics on $S^4$ and $\C P^2$ are discussed in \Cref{sec:s4andcp2}. \Cref{sec:sec>0gs} contains the proof of \Cref{mainthmB}, focusing mainly on the case of $S^4$, since the proof for $\C P^2$ is mostly analogous. Finally, \Cref{mainthmA} is proved in \Cref{sec:pos-neg}.


\section{Preliminaries}\label{sec:prelim}

\subsection{Cohomogeneity one}
We briefly discuss the geometry of cohomogeneity one manifolds to fix notations, see \cite{mybook,bettiol-krishnan1,grove-ziller-annals,gz-inventiones,VZ18,ziller-coh1survey} for details.

A cohomogeneity one manifold is a Riemannian manifold $(M,\gm)$ endowed with an isometric action by a Lie group $\G$, such that the orbit space $M/\G$ is one-dimensional. 
Let $\pi\colon M\to M/\G$ be the projection map. Throughout, we assume $M/\G=[0,L]$ is a closed interval, and the nonprincipal orbits $B_-=\pi^{-1}(0)$ and $B_+=\pi^{-1}(L)$ are \emph{singular orbits}. In other words, $B_\pm$ are smooth submanifolds of dimension strictly smaller than the principal orbits $\pi^{-1}(r)$, $r\in (0,L)$, which are smooth hypersurfaces of $M$.
Fix $x_{-}\in B_{-}$, and consider a minimal geodesic $\gamma(r)$ in $M$ joining $x_{-}$ to $B_{+}$, meeting it at $x_{+}=\gamma(L)$; that is, $\gamma$ is a horizontal lift of $[0,L]$ to $M$. Denote by $\K_{\pm}$ the isotropy group at $x_{\pm}$, and by $\H$ the isotropy at $\gamma(r)$, for $r\in (0,L)$. By the Slice Theorem, given $r_{\mathrm{max}}^\pm>0$ so that $r_{\mathrm{max}}^+ +r_{\mathrm{max}}^-=L$, 
the tubular neighborhoods $D(B_{-}) = \pi^{-1}\left(\left[0, r_{\mathrm{max}}^-\right]\right)$ and $D(B_{+}) = \pi^{-1}\left(\left[L-r_{\mathrm{max}}^+ ,L\right]\right)$ of the singular orbits are disk bundles over $B_-$ and $B_+$. Let $D^{l_{\pm}+1}$ be the normal disks to $B_{\pm}$ at $x_{\pm}$. Then $\K_{\pm}$ acts transitively on the boundary $\partial D^{l_{\pm}+1}$, with isotropy $\H$, so $\partial D^{l_{\pm}+1} = S^{l_{\pm}} = \K_{\pm}/\H$, and the $\K_{\pm}$-action on $\partial D^{l_{\pm}+1}$ extends to a $\K_\pm$-action on all of $D^{l_{\pm}+1}$. Moreover, there are equivariant diffeomorphisms $D(B_{\pm}) \cong\G\times_{\K_{\pm}}D^{l_{\pm}+1}$, and $M\cong D(B_-)\cup D(B_+)$, where the latter is given by gluing these disk bundles along their common boundary $\partial D(B_{\pm}) \cong\G\times_{\K_{\pm}} S^{l_{\pm}}\cong \G/\H$.
In this situation, one associates to $M$ the \emph{group diagram} 
\begin{equation*}
\H\subset\{\K_-,\K_+\}\subset \G.    
\end{equation*}
Conversely, given a group diagram as above, where $\K_{\pm}/\H$ are spheres, there exists a cohomogeneity one manifold $M$ given as the union of the above disk bundles.

Fix a bi-invariant metric $Q$ on the Lie algebra $\g$ of $\G$, and set $\n = \h^\perp$, where $\h\subset\g$ is the Lie algebra of $\H$. Identifying $\n\cong T_{\gamma(r)}(\G/\H)$ for each $0<r<L$ via action fields 
$X\mapsto X^*_{\gamma(r)}$, 
any $\G$-invariant metric on $M$ can be written~as
\begin{equation}\label{eqn:coh1metric}
 \gm =\dd r^2 + \gm_r, \quad  0<r<L,  
\end{equation}
along the geodesic $\gamma(r)$, where $\gm_r$ is a $1$-parameter family of left-invariant metrics on $\G/\H$, i.e., of $\Ad(\H)$-invariant metrics on $\n$. As $r\searrow0$ and $r\nearrow L$, the metrics $\gm_r$ degenerate, according to how $\G(\gamma(r))\cong\G/\H$ collapse to $B_\pm=\G/\K_\pm$. Namely, they satisfy \emph{smoothness conditions} that characterize when a tensor defined by means of \eqref{eqn:coh1metric} on $M\setminus (B_-\cup B_+)\cong (0,L)\times \G/\H$  extends smoothly to all of $M$, see \cite{VZ18}.

\subsubsection{Grove--Ziller metrics}\label{subsec:GZ-metrics}
If both singular orbits $B_\pm$ of a cohomogeneity one manifold $M$ have codimension two, then $M$ can be endowed with a new $\G$-invariant metric $\gGZ$ with $\sec\geq0$, as shown in the celebrated work of Grove and Ziller~\cite[Thm.~2.6]{grove-ziller-annals}. 
We now describe this construction, building metrics with $\sec\geq0$ on each disk bundle $D(B_\pm)$ that restrict to a fixed product metric $\dd r^2+b^2 Q|_{\mathfrak n}$ near $\partial D(B_\pm)\cong \G/\H$, so that these two pieces can be isometrically glued together.

Consider one such disk bundle $D(B)$ at a time, say over a singular orbit $B=\G/\K$, and let $\k$ be the Lie algebra of $\K$.
Set $\m = \k^\perp$ and $\p = \h^\perp \cap \k$, so that $\g=\m\oplus\p\oplus\h$ is a $Q$-orthogonal direct sum. Since $\p$ is $1$-dimensional, the metric $Q_{a,b}$ on $\G$, given~by 
\begin{align*}
  Q_{a,b}|_\m := b^2\,  Q|_\m, \qquad Q_{a,b}|_\p := ab^2\, Q|_\p, \qquad Q_{a,b}|_\h :=b^2\, Q|_\h,
\end{align*}
has $\sec \geq 0$ whenever $0<a \leq \frac{4}{3}$ and $b>0$, see \cite[Prop.~2.4]{grove-ziller-annals} or \cite[Lemma~3.2]{bm-mathann}.
Fix $1<a\leq \frac43$, and let $r_{\mathrm{max}}>0$ be such that
\begin{equation}\label{eq:MVT-obstruction}
y:=\tfrac{ \rho\sqrt{a}}{\sqrt{a-1}} \;\;\text{ satisfies }\;\; y< \, r_{\mathrm{max}},   
\end{equation}
where $\rho=\rho(b)$ is the radius of the circle(s) $\K/\H$ endowed with the metric $b^2\,Q|_{\p}$. 
Then, we can find a smooth nondecreasing function $f\colon \left[0,r_{\mathrm{max}}\right]\to\R$ and some $0<r_0<r_{\mathrm{max}}$, with $f(0)=0$, $f'(0)=1$, $f^{(2n)}(0)=0$ for all $n\in\mathds N$, $f''(r)\leq 0$ for all $r\in \left[0,r_{\mathrm{max}}\right]$, $f^{(3)}(r) > 0$ for all $r\in [0, r_0)$, and $f(r) \equiv y$ for all $r\in \left[r_0, r_{\mathrm{max}}\right]$.
The rotationally symmetric metric $\gm_{D^2} = \dd r^2 + f(r)^2 \dd\theta^2$, $0<r\leq r_{\mathrm{max}}$, on the punctured disk $D^2\setminus\{0\}$  extends to a smooth metric $\gm_{D^2}$ on $D^2$ with $\sec\geq0$ that, near $\partial D^2=\{r=r_{\mathrm{max}}\}$, is isometric to a round cylinder $\left[r_0, r_{\mathrm{max}}\right]\times S^1(y)$ of radius $y$.
Thus, the product manifold $(\G\times D^2, Q_{a,b} + \gm_{D^2})$ has $\sec \geq 0$, and so does the orbit space $D(B)\cong \G\times_\K D^2$ of the $\K$-action on $\G\times D^2$, when endowed with the metric $\gGZ$ that makes the projection map $\Pi\colon (\G\times D^2, Q_{a,b} + \gm_{D^2})\to (\G\times_\K D^2,\gGZ)$ a Riemannian submersion. Writing this metric $\gGZ$ in the form \eqref{eqn:coh1metric}, we have
\begin{equation}\label{eq:gGZcoh1form}
    \gGZ=\dd r^2+b^2\, Q|_{\m} +\tfrac{f(r)^2a}{f(r)^2 + a \rho^2}b^2\,Q|_{\p}, \quad 0<r\leq r_{\mathrm{max}},
\end{equation}
see e.g.~\cite[Lemma~2.1, Rem.~2.7]{grove-ziller-annals} or \cite[Lemma 3.1~(ii)]{bm-mathann}.
In particular, $\gGZ=\dd r^2 +b^2\, Q|_{\n}$ for all $r\in \left[r_0, r_{\mathrm{max}}\right]$, since $\tfrac{f(r)^2a}{f(r)^2 + a \rho^2}\equiv 1$ for all such $r$; hence $(D(B),\gGZ)$ is isometric to the prescribed product metric near $\partial D(B)\cong \G/\H$. 

This construction can be performed on each disk bundle $D(B_\pm)$ with 
the same $b>0$, provided $r_{\mathrm{max}}^\pm>0$ are chosen sufficiently large so that \eqref{eq:MVT-obstruction} holds for the corresponding radii $\rho_\pm(b)$ of the circles $\K_\pm/\H$ endowed with the metric $b^2\,Q|_{\mathfrak p_\pm}$.
Gluing these two disk bundles together, we obtain the desired $\G$-invariant metric $\gGZ$ with $\sec\geq0$ on $M\cong D(B_-)\cup D(B_+)$ and $M/\G=[0,L]$, where $L = r_{\mathrm{max}}^+ + r_{\mathrm{max}}^-$.
Although it is natural to pick the same (largest) value for $r_{\mathrm{max}}^\pm$, so that the gluing occurs at $r=\frac{L}{2}$, it is convenient to not impose this restriction.
Note that 
\begin{equation}\label{eq:lowerboundL}
    L  = r_{\mathrm{max}}^+ + r_{\mathrm{max}}^- > 
    \tfrac{\sqrt{a}}{\sqrt{a-1}}\,\big(\rho_+(b) +\rho_-(b)\big),
\end{equation}
if the gluing interface $\partial D(B_\pm)$ is isometric to $(\G/\H,b^2 Q|_\n)$.
Conversely, given $1<a\leq\frac43$, $b>0$, and $L$ satisfying \eqref{eq:lowerboundL}, there exists a Grove--Ziller metric on $M$ with gluing interface $(\G/\H,b^2 Q|_\n)$, induced by $Q_{a,b}+\gm_{D^2}$, and with $M/\G=[0,L]$.


\begin{remark}\label{rem:psec0r1]}
Although this is not a requirement in the original Grove--Ziller construction, we assume that $f^{(3)}(r) > 0$ on $[0, r_0)$, hence the curvature of $(D^2, \gm_{D^2})$ is monotonically decreasing for $r\in [0,r_0)$.  As a consequence, for each $0<r_*<r_0$, there is a constant $c>0$, depending on $r_*$, so that $\sec_{\gm_{D^2}}  \geq c$ for all $r\in [0,r_*]$.
\end{remark}

\subsection{Finsler--Thorpe trick}
In order to verify $\sec>0$ on Riemannian $4$-manifolds, we shall use a result that became known in the Geometric Analysis community as \emph{Thorpe's trick}, attributed to Thorpe~\cite{Thorpe72}, but that actually follows from much earlier work of Finsler~\cite{finsler}, and is often referred to 
as \emph{Finsler's Lemma} in Convex Algebraic Geometry.
This rather multifaceted result is also known as the \emph{$S$-lemma}, or \emph{$S$-procedure}, in the mathematical optimization and control literature, see e.g.~\cite{slemma-survey}. Details and other geometric perspectives can be found in \cite{bkm-siaga}.

Let $\Sym^2_{\mathrm b}(\wedge^2\R^n)\subset \Sym^2(\wedge^2\R^n)$ be the subspace of symmetric endomorphisms $R\colon \wedge^2\R^n\to\wedge^2\R^n$ that satisfy the first Bianchi identity. These objects are called \emph{algebraic curvature operators}, and serve as pointwise models for the curvature operators of Riemannian $n$-manifolds.
For instance, $R\in\Sym^2_{\mathrm b}(\wedge^2\R^n)$ is said to have $\sec\geq0$, respectively $\sec>0$, if the restriction of the quadratic form $\langle R(\sigma),\sigma\rangle$ to the oriented Grassmannian $\operatorname{Gr}_2^+(\R^n)\subset\wedge^2\R^n$ of $2$-planes is nonnegative, respectively positive. A Riemannian manifold $(M^n,\gm)$ has $\sec\geq0$, or $\sec>0$, if and only if its curvature operator $R_p\in\Sym^2_{\mathrm b}(\wedge^2 T_pM)$ has $\sec\geq0$, or $\sec>0$, for all $p\in M$.

The orthogonal complement to $\Sym^2_{\mathrm b}(\wedge^2\R^n)$ is identified with $\wedge^4\R^n$; so, if $n=4$, it 
is $1$-dimensional, and spanned by the Hodge star operator $*$. 
Since $\sigma\in\wedge^2\R^4$ satisfies $\sigma\wedge\sigma=0$ if and only if $\langle *\sigma,\sigma\rangle=0$, the quadric defined by $*$ in $\wedge^2\R^4$ is precisely the Pl\"ucker embedding $\operatorname{Gr}_2^+(\R^4)\subset\wedge^2\R^4$. 
As shown by Finsler~\cite{finsler}, a quadratic form $\langle R(\sigma),\sigma\rangle$ is nonnegative when restricted to the quadric $\langle *\sigma,\sigma\rangle=0$ if and only if some linear combination of $R$ and $*$ is positive-semidefinite, yielding:

\begin{proposition}[Finsler--Thorpe trick]\label{prop:FTtrick}
Let $R\in \Sym^2_{\mathrm b}(\wedge^2 \R^4)$ be an algebraic curvature operator. Then $R$ has $\sec\geq0$, respectively $\sec>0$, if and only if there exists $\tau\in\R$ such that $R+\tau\, *\succeq0$, respectively $R+\tau\, *\succ0$.
\end{proposition}

\begin{remark}\label{rem:set-of-taus}
For a given $R\in \Sym^2_{\mathrm b}(\wedge^2 \R^4)$ with $\sec\geq0$, the set of $\tau\in\R$ such that $R+\tau\,*\succeq0$ is a closed interval $[\tau_{\mathrm{min}},\tau_{\mathrm{max}}]$, which degenerates to a single point, i.e., $\tau_{\mathrm{min}}=\tau_{\mathrm{max}}$, if and only if $R$ does not have $\sec>0$, see \cite[Prop.~3.1]{bkm-siaga}
\end{remark}

The equivalences given by Finsler--Thorpe's trick offer substantial computational advantages to test for $\sec\geq0$ or $\sec>0$, see the discussion in~\cite[Sec.~5.4]{bkm-siaga}.

\section{\texorpdfstring{Cohomogeneity one structure of $S^4$ and $\C P^2$}{Cohomogeneity one structure of the sphere and complex projective plane}}\label{sec:s4andcp2}

Both $S^4$ and $\C P^2$ admit a cohomogeneity one action by $\G=\SO(3)$ as we now recall, see~\cite[Sec.~3]{bettiol-krishnan1} and \cite[Sec.~2]{ziller-coh1survey} for details. The $\G$-action on $S^4$ is the restriction to the unit sphere of the $\SO(3)$-action by conjugation on the 
space of symmetric traceless $3\times 3$ real matrices, while the $\G$-action on $\C P^2$ is a subaction of the transitive $\SU(3)$-action.
The corresponding orbit spaces are $S^4/\G=\left[0,\frac\pi3\right]$ and $\C P^2/\G=\left[0,\frac\pi4\right]$, endowing $S^4$ with the round metric with $\sec\equiv1$, and $\C P^2$ with the Fubini--Study metric with $1\leq\sec\leq4$.
Their group diagrams are as follows:
\begin{align*}
   S^4&:  &\Z_2\oplus\Z_2\cong \S(\OO(1)\OO(1)\OO(1)) &\subset \{  \S(\OO(1)\OO(2)),\S(\OO(2)\OO(1))\} \subset\SO(3),\\
   \C P^2&: &\Z_2\cong \langle\operatorname{diag}(-1,-1,1) \rangle &\subset \{ \S(\OO(1)\OO(2)), \SO(2)_{1,2}\} \subset\SO(3),
\end{align*}
according to an appropriate choice of minimal geodesic $\gamma(r)$, $r\in [0,L]$, see~\cite[Sec.~3]{bettiol-krishnan1}. In both cases, since $\H$ is discrete, $\n \cong\g= \mathfrak{so}(3)$. 
We henceforth fix $Q$ to be the bi-invariant metric such that $\{E_{23}, E_{31}, E_{12}\}$ is a $Q$-orthonormal basis of $\mathfrak{so}(3)$, where $E_{ij}$ is the skew-symmetric $3\times 3$ matrix with a $+1$ in the $(i,j)$ entry, a $-1$ in the $(j,i)$ entry, and zeros in the remaining entries.
The $1$-dimensional subspaces $\n_k=\operatorname{span}(E_{ij})$, where $(i,j,k)$ is a cyclic permutation of $(1,2,3)$, are pairwise inequivalent for the adjoint action of $\H$ in the case of $S^4$, while $\n_1$ and $\n_2$ are equivalent in the case of $\C P^2$, but neither is equivalent to $\n_3$.

Collectively denoting $S^4$ and $\C P^2$ with the above cohomogeneity one structures by $M^4$, 
we consider \emph{diagonal} $\G$-invariant metrics $\gm$ on $M^4$, i.e., metrics of the form
\begin{equation}\label{eq:g-phi,psi,xi}
 \gm =\dd r^2+ \varphi(r)^2 \, Q|_{\n_1} + \psi(r)^2 \, Q|_{\n_2} + \xi(r)^2 \, Q|_{\n_3}, \quad 0<r<L,
\end{equation}
where $L=\frac\pi3$ or $L=\frac\pi4$ according to whether $M^4=S^4$ or $M^4=\C P^2$, cf.~\eqref{eqn:coh1metric}. Note that every $\G$-invariant metric on $S^4$ is of the above form, i.e., 
$\n_k$ are pairwise orthogonal, but $\n_1$ and $\n_2$ need not be orthogonal for all $\G$-invariant metrics on $\C P^2$, i.e., the off-diagonal term $\gm(E_{23},E_{31})$ need not vanish identically. 
The standard metric on $M^4$, with curvatures normalized as above, is obtained setting $\varphi,\psi,\xi$ to
\begin{equation}\label{eq:can-phi,psi,xi}
\begin{aligned}
   S^4&:  &\varphi_1(r)&=2\sin r,  &\psi_1(r)&= \sqrt{3}\cos r + \sin r,  &\xi_1(r)&= \sqrt{3}\cos r - \sin r, \\
   \C P^2&: &\varphi_1(r)&=\sin r,  &\psi_1(r)&= \cos r,  &\xi_1(r)&= \cos 2r,
\end{aligned}
\end{equation}
see \Cref{fig:canmetrics} below for their graphs.

\subsection{Smoothness}
The conditions required of $\varphi, \psi, \xi$ for the metric $\gm$ in \eqref{eq:g-phi,psi,xi}, which is defined on the open dense set $M^4\setminus (B_-\cup B_+)\cong (0,L)\times \G/\H$, to extend smoothly to all of $M^4$ can be extracted from~\cite[Sec.~3.1, 3.2]{VZ18} as follows:

\begin{proposition}\label{prop:smoothness}
The $\G$-invariant metric \eqref{eq:g-phi,psi,xi} on $M^4\setminus (B_-\cup B_+)$ extends to a smooth metric on $M^4$ if and only if $\varphi,\psi,\xi$ extend smoothly to $r=0$ and $r=L$ satisfying the following, where $\phi_k$ are smooth, $z=L-r$, and $\varepsilon>0$ is small:

\smallskip
\begin{center}
\begin{tabular}{|c|l|}
\hline
$M^4$ & \rule[-1.2ex]{0pt}{0pt} \rule{0pt}{2.2ex} Smoothness conditions on $\varphi,\psi,\xi$
\\
\hline \noalign{\medskip} \hline 
$\begin{array}{c}
S^4 \\[5pt]
L =\frac\pi3
\end{array}$  &  $\begin{array}{l}
          {\rm (i)} \; \varphi(0) = 0,\,  \rule{0pt}{2.5ex} \varphi'(0) = 2, \,\varphi^{(2n)}(0) = 0, \text{  for all } n \geq 1, \\[1pt] 
          {\rm (ii)} \;\psi(r)^2 + \xi(r)^2 = \phi_1(r^2), \text{  for all } r\in [0,\varepsilon), \\[1pt] 
          {\rm (iii)} \;\psi(r)^2 - \xi(r)^2 = r\,\phi_2(r^2), \text{  for all } r\in [0,\varepsilon), \\[3pt]
          {\rm (iv)} \;\xi(L) = 0, \, \xi'(L) = -2, \, \xi^{(2n)}(L) = 0, \text{  for all } n \geq 1, \\[1pt]
          {\rm (v)} \;\psi(z)^2 + \varphi(z)^2 = \phi_3(z^2), \text{  for all } z\in [0,\varepsilon), \\[1pt] 
          {\rm (vi)} \;\psi(z)^2 - \varphi(z)^2 = z\,\phi_4(z^2), \text{  for all } z\in [0,\varepsilon).   
        \end{array}$   \\
\hline \noalign{\smallskip} \hline 
$\begin{array}{c}
\C P^2 \\[5pt]
L =\frac\pi4
\end{array}$  & $\begin{array}{l}
         {\rm (i)} \; \varphi(0) = 0,\,  \rule{0pt}{2.5ex} \varphi'(0) = 1, \,\varphi^{(2n)}(0) = 0, \text{  for all } n \geq 1, \\[1pt] 
          {\rm (ii)} \;\psi(r)^2 + \xi(r)^2 = \phi_5(r^2), \text{  for all } r\in [0,\varepsilon), \\[1pt] 
         {\rm (iii)} \; \psi(r)^2 - \xi(r)^2 = r^2\,\phi_6(r^2), \text{  for all } r\in [0,\varepsilon), \\[3pt]
          {\rm (iv)} \; \xi(L) = 0, \, \xi'(L) = -2, \, \xi^{(2n)}(L) = 0, \text{  for all } n \geq 1, \\[1pt]
         {\rm (v)} \;  \psi(z)^2 + \varphi(z)^2 = \phi_7(z^2), \text{  for all } z\in [0,\varepsilon), \\[1pt] 
         {\rm (vi)} \;  \psi(z)^2 - \varphi(z)^2 = z\,\phi_8(z^2), \text{  for all } z\in [0,\varepsilon). 
        \end{array}$ \\
\hline
\end{tabular}
\end{center}
\end{proposition}

\begin{remark}\label{rem:extra-symm}
Since the isotropy groups $\K_\pm$ for the $\G$-action on $S^4$ are conjugate,  the smoothness conditions at the endpoints $r=0$ and $r=L$ can be obtained from one another by interchanging the roles of $\varphi$ and $\xi$. 
Furthermore, just as the round metric \eqref{eq:can-phi,psi,xi}, all metrics we consider on $S^4$ have the following additional symmetries:
\begin{equation}\label{eqn:gS4sym}
 \varphi(r)=\xi\left(L-r\right), \quad \text{and}\quad \psi(r)=\psi\left(L-r\right), \quad \text{ for all } 0\leq r\leq L.
\end{equation}
However, metrics on $\C P^2$ do not have any of these features or extra symmetries, as $\K_\pm$ are not conjugate, and, in general
$\varphi(r) \neq \xi\left(L-r\right)$ and $\psi(r)\neq\psi\left(L-r\right)$. 
\end{remark}

\begin{figure}[!ht]
\begin{tikzpicture}[scale=0.75]
\begin{axis}[
        axis x line=middle, 
        axis y line=middle, axis line style = thick, tick style = thick,
        ymax=2.1, xmax={pi/3+.1},
        xtick={pi/6,pi/3},
        xticklabels={$\frac\pi6$, $\frac\pi3$},
        ytick={sqrt(3)},
        yticklabels={$\sqrt3$},
        ]
    \addplot[domain=0:pi/3, red, ultra thick] {2*sin(deg(x))}; 
    \addplot[domain=0:pi/3, black!40!green, ultra thick] {sqrt(3)*cos(deg(x))+sin(deg(x))}; 
    \addplot[domain=0:pi/3, blue, ultra thick] {sqrt(3)*cos(deg(x))-sin(deg(x))}; 
    \draw (pi/4,1.55) node {$\color{red}\varphi_1$};
    \draw (pi/4,2.05) node {$\color{black!40!green}\psi_1$};
    \draw (pi/4,0.69) node {$\color{blue}\xi_1$};
\end{axis}
\end{tikzpicture}
\begin{tikzpicture}[scale=0.75]
\begin{axis}[
        axis x line=middle, 
        axis y line=middle, axis line style = thick, tick style = thick,
        ymax=2.1,   xmax={pi/4+.1},
        xtick={pi/8,pi/6,pi/4},
        xticklabels={$\frac\pi8$, $\frac{\pi}{6}$, $\frac\pi4$},
        ytick={sqrt(2)/2, 1},
        yticklabels={$\frac{\sqrt2}{2}$, $1$},
        ]
    \addplot[domain=0:pi/4, red, ultra thick] {sin(deg(x))}; 
    \addplot[domain=0:pi/4, black!40!green, ultra thick] {cos(deg(x))}; 
    \addplot[domain=0:pi/4, blue, ultra thick] {cos(2*deg(x))}; 
    \draw (pi/5,0.67) node {$\color{red}\varphi_1$};
    \draw (pi/5,0.9) node {$\color{black!40!green}\psi_1$};
    \draw (pi/5,0.42) node {$\color{blue}\xi_1$};
\end{axis}
\end{tikzpicture}
\vspace{-.2cm}
\caption{Graphs of $\varphi_1,\psi_1,\xi_1$, for $S^4$ (left) and $\C P^2$ (right).}\label{fig:canmetrics}
\vspace{-.2cm}
\end{figure}
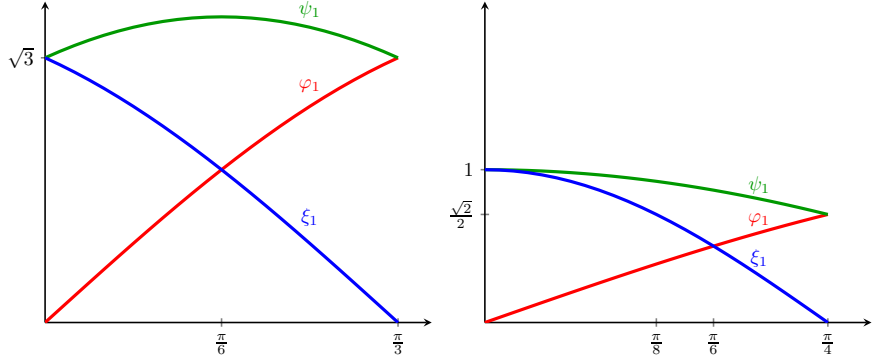

\subsection{Curvature}
Computing the curvature operator of the $\G$-invariant metric \eqref{eq:g-phi,psi,xi} on $M^4$, with the formulae in \cite[Prop.~1.12]{gz-inventiones}, one obtains the following:

\begin{proposition}\label{propn:curv_op}
Let $\{ e_i \}_{i=0}^3$ be the $\gm$-orthonormal frame along the geodesic $\gamma(r)$, $0<r<L$, given by $e_0=\gamma'(r)$, $e_1 = \frac{1}{\varphi(r)} E_{23}^*$, $e_2 = \frac{1}{\psi(r)} E_{31}^*$, $e_3 = \frac{1}{\xi(r)} E_{12}^*$, i.e., $e_0$ is the unit horizontal direction and $\{e_1,e_2,e_3\}$ are unit Killing vector fields. In the basis 
$\B:=\{ 
e_2\wedge e_3,\, e_0\wedge e_1,\, 
e_3\wedge e_1,\, e_0\wedge e_2,\,
e_1\wedge e_2,\, e_0\wedge e_3
\},$
the curvature operator $\RR\colon\wedge^2 T_{\gamma(r)}M^4 \to\wedge^2 T_{\gamma(r)}M^4$, $0<r<L$, is block diagonal, that is, $\RR = \operatorname{diag}(\RR_1, \RR_2, \RR_3)$, with $2\times 2$ blocks given as follows:
\begin{align*}
 \RR_1 &= \begin{bmatrix}
  \frac{\psi^4+\xi^4 -\varphi^4 + 2(\xi^2-\varphi^2)(\varphi^2-\psi^2)}{4\varphi^2  \psi^2\xi^2 } - \frac{\psi'\xi'}{\psi\xi }   &
  \; \frac{\psi'(\psi^2+\varphi^2-\xi^2)}{2 \varphi\psi^2\xi} + \frac{\xi'(\xi^2+\varphi^2-\psi^2)}{2\varphi\psi\xi^2} -\frac{\varphi'}{\psi\xi}   \\[5pt]
 \; \frac{\psi'(\psi^2+\varphi^2-\xi^2)}{2 \varphi\psi^2\xi} + \frac{\xi'(\xi^2+\varphi^2-\psi^2)}{2\varphi\psi\xi^2} -\frac{\varphi'}{\psi\xi} &
  -\frac{\varphi''}{\varphi}
 \end{bmatrix},\\[3pt]
 \RR_2 &= \begin{bmatrix}
  \frac{\varphi^4 + \xi^4-\psi^4 + 2(\varphi^2-\psi^2)(\psi^2-\xi^2)}{4\varphi^2 \psi^2\xi^2} - \frac{\varphi'\xi'}{\varphi \xi}   &
  \;  \frac{\varphi'(\varphi^2+\psi^2-\xi^2)}{2\varphi^2\psi \xi } + \frac{\xi'(\xi^2+\psi^2-\varphi^2)}{2\varphi \psi\xi^2} -\frac{\psi'}{\varphi \xi} \\[5pt]
  \; \frac{\varphi'(\varphi^2+\psi^2-\xi^2)}{2\varphi^2\psi \xi } + \frac{\xi'(\xi^2+\psi^2-\varphi^2)}{2\varphi \psi\xi^2} -\frac{\psi'}{\varphi \xi}  &
  -\frac{\psi''}{\psi}
 \end{bmatrix},\\[3pt]
 \RR_3 &= \begin{bmatrix}
  \frac{ \varphi^4+\psi^4-\xi^4 + 2(\psi^2-\xi^2)(\xi^2-\varphi^2)}{4\varphi^2 \psi^2\xi^2 } - \frac{\varphi'\psi'}{\varphi\psi}   &
   \frac{\varphi'(\varphi^2+\xi^2-\psi^2)}{2\varphi^2\psi\xi}+ \frac{\psi'(\psi^2+\xi^2-\varphi^2)}{2\varphi\psi^2 \xi}  -\frac{\xi'}{\varphi\psi}  \\[5pt]
    \frac{\varphi'(\varphi^2+\xi^2-\psi^2)}{2\varphi^2\psi\xi}+ \frac{\psi'(\psi^2+\xi^2-\varphi^2)}{2\varphi\psi^2 \xi}  -\frac{\xi'}{\varphi\psi}  &
  -\frac{\xi''}{\xi}
 \end{bmatrix}.
\end{align*}
\end{proposition}

The Hodge star operator $*$ is also clearly block diagonal in the basis $\B$, namely,
\begin{equation}\label{eq:defH}
* = \operatorname{diag}(H, H, H), \quad\text{where}\quad H= \begin{bmatrix}
      0 & 1\\
      1 & 0
     \end{bmatrix}.
\end{equation}
Thus, by the Finsler--Thorpe trick (\Cref{prop:FTtrick}), such  $\RR=\diag(\RR_1,\RR_2,\RR_3)$ as in \Cref{propn:curv_op} has $\sec\geq0$, respectively $\sec>0$, if and only if there exists $\tau(r)$ such that
$\RR_i+\tau\,H\succeq0$ for $i=1,2,3$, respectively $\RR_i+\tau\,H\succ0$ for $i=1,2,3$.

\begin{remark}
Diagonal entries in  $\RR_i$ are sectional curvatures $\sec(e_i\wedge e_j)=R_{ijij}$ of coordinate planes, while off-diagonal entries are $R_{ijkl}$, with $i,j,k,l$ all distinct, so the Finsler--Thorpe trick states that $\sec\geq0$ and $\sec>0$ are respectively equivalent to the existence of $\tau$ such that all $R_{ijij}\,R_{klkl}- (R_{ijkl}+\tau)^2$ are $\geq0$ and $>0$.
\end{remark}

To illustrate the above, note that setting $\varphi,\psi,\xi$ to be the functions in \eqref{eq:can-phi,psi,xi} that correspond to the standard metrics in $S^4$ and $\C P^2$, the blocks $R_i$ become constant:
\begin{equation}\label{eq:can-R-blocks}
\begin{aligned}
   S^4&: \qquad \RR_1=\RR_2=\RR_3=\begin{bmatrix}
           1 & 0 \\
           0 & 1
          \end{bmatrix},  \\[1pt]
   \C P^2&: \qquad \RR_1 = \RR_2= \begin{bmatrix}
           1 & -1 \\
           -1 & 1
          \end{bmatrix}, \quad \RR_3 = \begin{bmatrix}
           4 & 2 \\
           2 & 4
          \end{bmatrix}.
\end{aligned}
\end{equation}
In particular, $\tau$ can be chosen constant, and 
 $R+\tau\,*\succeq0$ if and only if $\tau\in[-1,1]$ for $S^4$, and $\tau\in[0,2]$ for $\C P^2$, and $R+\tau\,*\succ0$ if and only if $\tau$ is in the open intervals.

Similarly, the curvature of a Grove--Ziller metric 
with gluing interface $\partial D(B_\pm)$ isometric to $(\G/\H,b^2Q|_\n)$ and $L=r_{\mathrm{max}}^+ +r_{\mathrm{max}}^-$ can be computed by setting $\varphi,\psi,\xi$ instead to be the functions that make \eqref{eq:g-phi,psi,xi} match with \eqref{eq:gGZcoh1form}, namely (see \Cref{fig:GZmetrics})
\begin{align}
 \varphi(r) &= \begin{cases}
 \frac{f(r)\,b \,\sqrt{a}}{\sqrt{f(r)^2 + a \rho^2}}, & \text{ if }  r\in \left(0,r_{\mathrm{max}}^- \right], \text{ where } \rho=\rho_-(b), \; f=f_-, \\ 
 b, & \text{ if } r\in \left[r_{\mathrm{max}}^- , L\right),
 \end{cases}\nonumber \\[3pt]
 \psi(r)&\equiv b,\label{eqn:f1intermsoff} \\[3pt]
 \xi(r) &= \begin{cases}
b, & \text{ if }  r\in \left(0,r_{\mathrm{max}}^- \right],\\
 \frac{f(L-r)\,b \,\sqrt{a}}{\sqrt{f(L-r)^2 + a \rho^2}},& \text{ if }r \in \left[r_{\mathrm{max}}^- , L\right), \text{ where } \rho=\rho_+(b), \; f=f_+,
 \end{cases}\nonumber
 \end{align}
as $\m=\n_2\oplus \n_3$ and $\p=\n_1$ for the disk bundle $D(B_-)$, but $\varphi$ and $\xi$ switch roles on the disk bundle $D(B_+)$, in which $\m=\n_1\oplus \n_2$ and $\p= \n_3$.
Recall that $f(r)\equiv \tfrac{\sqrt{a}\,\rho}{\sqrt{a-1}}$ for $r_0\leq r\leq r_{\mathrm{max}}$ on each of $D(B_\pm)$, so, in a neighborhood of the gluing interface $r=r_{\mathrm{max}}^-=L-r_{\mathrm{max}}^+$, the functions $\varphi=\psi=\xi$ are all constant and equal to $b$ . 

In what follows, to simplify the exposition, we shall work with $\varphi,\psi,\xi$ 
only on the interval $\left(0,r_{\mathrm{max}}^-\right]$,
which, at least on $S^4$, determines their values for all $0<r<L$ by setting $r_{\mathrm{max}}^+=r_{\mathrm{max}}^-$ and imposing the additional symmetries \eqref{eqn:gS4sym}, see \Cref{rem:extra-symm}. 

Straightforward computations using \Cref{propn:curv_op} imply the following:

\begin{proposition} \label{propn:GZRR}
The curvature operator of the Grove--Ziller metric \eqref{eq:gGZcoh1form}; i.e., the metric \eqref{eq:g-phi,psi,xi} with $\varphi,\psi,\xi$ as in \eqref{eqn:f1intermsoff}, for $r\in \left(0,r_{\mathrm{max}}^-\right]$, is
$R=\diag(R_1,R_2,R_3)$,~with:
 \begin{equation*}
 \RR_1 = \begin{bmatrix}
           \frac{4b^2 - 3\varphi^2}{4b^4} & -\frac{\varphi'}{b^2} \\
           -\frac{\varphi'}{b^2} & -\frac{\varphi''}{\varphi}
          \end{bmatrix}, \quad 
  \RR_2 = \RR_3 = \begin{bmatrix}
           \frac{\varphi^2}{4b^4} & \frac{\varphi'}{2b^2} \\
           \frac{\varphi'}{2b^2} & 0
          \end{bmatrix}.
 \end{equation*}
 In particular, $R+\tau\,*\succeq0$ if and only if $\tau=-\frac{\varphi'}{2b^2}$.
\end{proposition}

Indeed, it is easy to verify that $\tau=-\frac{\varphi'}{2b^2}$ is the \emph{only} function $\tau(r)$, $r\in \left(0,r_{\mathrm{max}}^-\right]$, such that $R+\tau\,*\succeq0$. 
Namely, for such $r$, we have that $[\RR_i + \tau H]_{22}\equiv 0$ for both $i=2,3$, and hence $\det(\RR_2 + \tau H)=-(\frac{\varphi'}{2b^2}+\tau)^2\geq0$. This pointwise uniqueness of $\tau$ corresponds to the presence of flat planes for the Grove--Ziller metric at every point $\gamma(r)$; e.g., $\sec(e_0\wedge e_2)\equiv 0$ for all $r$.
It is interesting to observe how this (forceful) choice of $\tau$ stemming from $R_i+\tau H\succeq0$, $i=2,3$, also satisfies $R_1+\tau H\succeq0$, i.e., how the expression for $\varphi$ in \eqref{eqn:f1intermsoff} ensures $\det(\RR_1 + \tau H) = \big( \frac{4b^2 - 3\varphi^2}{4b^4}\big)\big(-\frac{\varphi''}{\varphi}\big) - \big(\frac{3\varphi'}{2b^2}\big)^2\geq0$.

\begin{lemma}\label{propn:eqnf1_sec>0}
The function $\varphi(r)$ in the Grove--Ziller metric \eqref{eq:gGZcoh1form}, given by \eqref{eqn:f1intermsoff} for $r\in \left(0,r_{\mathrm{max}}^-\right]$, satisfies $(4b^2 - 3\varphi^2)(-\varphi'') - 9\varphi\varphi'^2 \geq 0$ for all $r\in \left(0,r_{\mathrm{max}}^-\right]$.
\end{lemma}

\begin{proof}
Solving for $f(r)$ in \eqref{eqn:f1intermsoff}, we find $f(r)= \frac{\varphi(r)\rho\sqrt{a}}{ \sqrt{ab^2 - \varphi(r)^2}}$; in particular, we have that $\varphi(r)<\sqrt{a}\,b$. Differentiating twice, it follows that:
\begin{equation}\label{eq:f''intermsofphi}
    f'' = \frac{a^{3/2}b^2\rho}{(ab^2 - \varphi^2)^{5/2}} \big( \varphi''(a b^2 - \varphi^2) + 3\varphi\varphi'^2\big).
\end{equation}
Since $f''\leq 0$, we have $\varphi''(a b^2 - \varphi^2) + 3\varphi\varphi'^2\leq0$, so $(3a b^2 - 3\varphi^2)(-\varphi'') - 9\varphi\varphi'^2 \geq 0$, which implies the desired differential inequality since $a \leq \frac{4}{3}$.
\end{proof}

\section{Positively curved metrics near Grove--Ziller metrics}\label{sec:sec>0gs}

In this section, we prove \Cref{mainthmB} in the Introduction, perturbing arbitrary Grove--Ziller metrics with $\sec\geq0$ on $S^4$ and $\C P^2$ into cohomogeneity one metrics that we show have $\sec>0$ via the Finsler--Thorpe trick (\Cref{prop:FTtrick}).

\subsection{Metric perturbation}
Let $M^4$ be either $S^4$ or $\C P^2$,  with the cohomogeneity one action of $\G=\SO(3)$ from the previous section. Given a Grove--Ziller metric $\gGZ$ on $M^4$ with gluing interface isometric to $(\G/\H,b^2 Q|_\n)$, we have
that the length of the circle(s) $\K_\pm/\H$ endowed with the metric $b^2\,Q|_{\p_\pm}$ is $\rho_\pm(b) = b/|(\K_\pm)_0\cap \H|$, where $\K_0$ is the identity component of $\K$. From the group diagrams, 
we  compute $|(\K_\pm)_0\cap \H|$ and obtain $\rho_\pm(b) = b/2$ if $M^4=S^4$, while $\rho_-(b) = b$ and $\rho_+(b) = b/2$ if $M^4=\C P^2$.
Thus, by \eqref{eq:lowerboundL}, the length $L$ of the orbit space $M/\G=[0,L]$ satisfies $L>\frac{\sqrt{a}}{\sqrt{a-1}} \,b$ if $M^4=S^4$, and $L> \frac{3\sqrt{a}}{2\sqrt{a-1}}\, b$ if $M^4=\C P^2$.
Rescaling $(M^4,\gGZ)$ so that $L=\frac\pi3$ if $M^4=S^4$, and $L = \frac\pi4$ if $M^4=\C P^2$, we obtain a Grove--Ziller metric $\gm_0$ homothetic to $\gGZ$, with standardized $L$, and whose parameters $a$ and $b$ satisfy
\begin{equation}\label{eq:max-beta}
\textstyle b< \frac\pi3\frac{\sqrt{a-1}}{\sqrt{a}} \;\text{ if }\; M^4=S^4, \quad\text{and}\quad b< \frac\pi6\frac{\sqrt{a-1}}{\sqrt{a}} \;\text{ if }\; M^4=\C P^2.
\end{equation}
Using \eqref{eq:MVT-obstruction}, it follows that $r_{\mathrm{max}}^\pm = \frac\pi6$ for $M^4=S^4$, while $r_{\mathrm{max}}^- = \frac{\pi}{6}$ and $r_{\mathrm{max}}^+ = \frac{\pi}{12}$ for $M^4=\C P^2$.
Note that $\varphi_1(r)=\xi_1(r)$ precisely at these values of $r=r^-_{\mathrm{max}}$.

Writing $\gm_0$ in the form \eqref{eq:g-phi,psi,xi} we obtain the functions $\varphi,\psi,\xi$ in \eqref{eqn:f1intermsoff}, which we decorate with the subindex $_0$, i.e., $\varphi_0,\psi_0,\xi_0$. Similarly, let $\gm_1$ be the standard metric on $M^4$, and use a subindex $_1$ to decorate the $\varphi,\psi,\xi$ given in \eqref{eq:can-phi,psi,xi}. Now, define:
\begin{equation}\label{eq:def-phis-psis-xis}
\begin{aligned}
 \varphi_s(r) &:= (1-s)\varphi_0(r) + s\,\varphi_1(r),\\
 \psi_s(r) &:=  (1-s)\psi_0(r) + s\,\psi_1(r), \qquad r\in \left[ 0, L \right],\\
  \xi_s(r) &:= (1-s)\,\xi_0(r) + s\,\xi_1(r),
\end{aligned}
\end{equation}
i.e., linearly interpolate from $\varphi_0,\psi_0,\xi_0$ to $\varphi_1,\psi_1,\xi_1$, and set $\gm_s$, $s\in[0,1]$, to be
\begin{equation}\label{eq:gs}
\gm_s := \dd r^2 + \varphi_s(r)^2 \, Q|_{\n_1} + \psi_s(r)^2 \, Q|_{\n_2} + \xi_s(r)^2 \, Q|_{\n_3}, \quad 0<r<L.
\end{equation}
The functions \eqref{eq:def-phis-psis-xis} can be visualized as affine homotopies between \Cref{fig:canmetrics,fig:GZmetrics}.

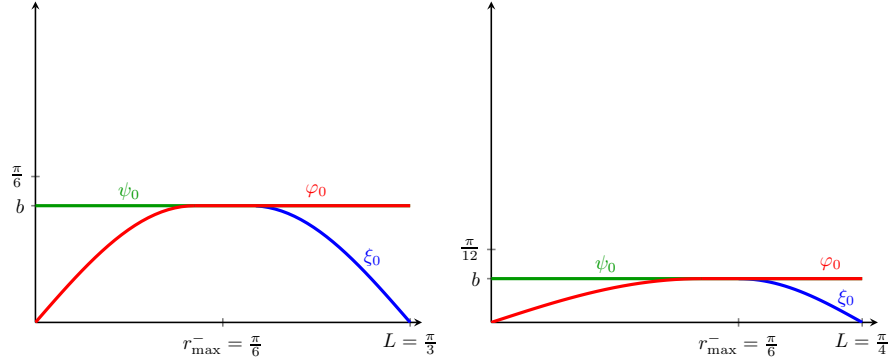
\begin{figure}[!ht]
\begin{tikzpicture}[scale=0.75]
\begin{axis}[
        axis x line=middle, 
        axis y line=middle, axis line style = thick, tick style = thick,
        ymax=2.2, xmax={pi+.1},
        xtick={pi/2,pi},
        xticklabels={$r_{\mathrm{max}}^-=\frac\pi6$, $L=\frac\pi3$},
        ytick={0.8,1},
        yticklabels={$b$,$\frac{\pi}{6}$},
        ]
    \addplot[domain=(pi/2)/.85:pi, blue, ultra thick] {0.8*sin(deg((pi-x)/.85))};
    \addplot[domain=0:pi, black!40!green, ultra thick] {0.8}; 
    \addplot[domain=0:pi/2*.85, red, ultra thick] {0.8*sin(deg(x/.85))}; 
    \addplot[domain=pi/2*.85:pi, red, ultra thick] {0.8}; 
    \draw (3*pi/4,0.9) node {$\color{red}\varphi_0$};
    \draw (pi/4,0.9) node {$\color{black!40!green}\psi_0$};
    \draw (9*pi/10,0.45) node {$\color{blue}\xi_0$};
\end{axis}
\end{tikzpicture}
\begin{tikzpicture}[scale=0.75]
\begin{axis}[
        axis x line=middle, 
        axis y line=middle, axis line style = thick, tick style = thick,
        ymax=2.2, xmax={3*pi/4+.1},
        xtick={pi/2,3*pi/4},
        xticklabels={$r_{\mathrm{max}}^-=\frac\pi6$, $L=\frac\pi4$},
        ytick={0.3,0.5},
        yticklabels={$b$,$\frac{\pi}{12}$},
        ]
    \addplot[domain=(3*pi/4-pi/2)/.5:3*pi/4, blue, ultra thick] {0.3*sin(deg((3*pi/4-x)/.5))}; 
    \addplot[domain=0:3*pi/4, black!40!green, ultra thick] {0.3}; 
    \addplot[domain=0:pi/2*.85, red, ultra thick] {0.3*sin(deg(x/.85))}; 
    \addplot[domain=pi/2*.85:3*pi/4, red, ultra thick] {.3}; 
    \draw (3*pi/4-.2,0.4) node {$\color{red}\varphi_0$};
    \draw (pi/5+.1,0.4) node {$\color{black!40!green}\psi_0$};
    \draw (3*pi/4-.1,0.17) node {$\color{blue}\xi_0$};
\end{axis}
\end{tikzpicture}
\vspace{-.2cm}
\caption{Graphs of $\varphi_0,\psi_0,\xi_0$, for $S^4$ (left) and $\C P^2$ (right), cf.~\eqref{eqn:f1intermsoff}. The  upper bound on $b$ and $r_{\mathrm{max}}^-=\frac\pi6$ follow from \eqref{eq:max-beta}.
}\label{fig:GZmetrics}
\vspace{-.1cm}
\end{figure}

It is a straightforward consequence of \Cref{prop:smoothness} that $\gm_s$ are smooth metrics:

\begin{lemma}\label{propn:gssmooth}
The $\G$-invariant metrics $\gm_s$, $s\in [0,1]$, defined on $M^4\setminus (B_-\cup B_+)$ by \eqref{eq:gs}, extend to smooth metrics on $M^4$, which we also denote by $\gm_s$, $s\in [0,1]$.
\end{lemma}

\begin{proof}
For simplicity, we focus on the case $M^4=S^4$, and the case $M^4=\C P^2$ is left to the reader. The metrics $\gm_s$ are clearly smooth away from the singular orbits, which correspond to $r=0$ and $r=L$. In light of \Cref{rem:extra-symm}, it suffices to check the smoothness conditions (i)--(iii) in \Cref{prop:smoothness}, i.e., those regarding $r=0$.

First, since $\varphi_s^{(k)}(r)=(1-s)\varphi_0^{(k)}(r)+s\,\varphi_1^{(k)}(r)$ for all $k\geq0$, it is clear that $\varphi_s$ satisfies (i), as both $\varphi_0$ and $\varphi_1$ do.
Second, if $r\in \left[0,r_{\mathrm{max}}^-\right]$, then $\psi_0(r)=\xi_0(r)=b$, cf.~\eqref{eqn:f1intermsoff}, so $\psi_s(r) = (1-s) b + s\,\psi_1(r)$ and $\xi_s(r) = (1-s)b + s\,\xi_1(r)$, and thus:
 \begin{align*}
\psi_s(r)^2 + \xi_s(r)^2 &= 2(1-s)^2b^2  + 2s(1-s)b (\psi_1(r) + \xi_1(r)) + s^2 \left( \psi_1(r)^2 + \xi_1(r)^2 \right)\\
  &= 2(1-s)^2b^2 + 2s(1-s)b\,\sqrt{3}\cos r + s^2 \phi_1(r^2) = \widetilde{\phi_1}(r^2),\\
\psi_s(r)^2 - \xi_s(r)^2  &=  2s(1-s)b \, (\psi_1(r) - \xi_1(r)) + s^2 \left( \psi_1(r)^2 - \xi_1(r)^2 \right) \\
  &=  2s(1-s)b\,(-2\sin r) + s^2 \,r\,\phi_2(r^2) = r\,\widetilde{\phi_2}(r^2),
 \end{align*}
 where $\widetilde{\phi_k}$, $k=1,2$, are smooth functions, hence (ii) and (iii) are also satisfied. 
\end{proof}

Let us introduce functions $\Delta_\varphi,\Delta_\psi,\Delta_\xi$ of $r$ so that \eqref{eq:def-phis-psis-xis} can be written as
\begin{equation}\label{eqn:Deltai}
\varphi_s=\varphi_0+s\,\Delta_\varphi, \quad \psi_s=\psi_0+s\,\Delta_\psi, \quad \xi_s=\xi_0+s\,\Delta_\xi,
\end{equation}
i.e., $\Delta_\varphi(r) := \varphi_1(r)-\varphi_0(r)$, and similarly for $\Delta_\psi$ and $\Delta_\xi$. Note that each of these functions is smooth up to $r=0$ and $r=L$; in particular, bounded on $[0,L]$. In the sequel, we take the point of view \eqref{eqn:Deltai} that $\varphi_s,\psi_s,\xi_s$ are perturbations of $\varphi_0,\psi_0,\xi_0$.

\subsection{Regularity of perturbation}
By \eqref{eq:gs}, \Cref{propn:gssmooth}, and \Cref{propn:curv_op}, each entry of the curvature operator matrix $R_s$ of $\gm_s$ along $\gamma(r)$ is a smooth function
\begin{equation}\label{eq:generic-entry}
 \frac{ P(\varphi_s,\, \psi_s,\,\xi_s,\, \varphi'_s,\, \psi'_s,\, \xi'_s,\, \varphi''_s,\, \psi''_s, \,\xi''_s)}{\varphi_s^2\,\psi_s^2\,\xi_s^2},
\end{equation}
where $P$ is a polynomial. 
Note that the $\gm_s$-orthonormal basis on which the matrix $R_s$ is being written varies smoothly with $s$. 
The singularities in \eqref{eq:generic-entry} at $r=0$ and $r=L$, due to $\varphi_s(0)=0$ and $\xi_s(L)=0$, are removable as a consequence of \Cref{propn:gssmooth}. This corresponds to the fact that also $P$ vanishes to the appropriate order because $\varphi_s,\psi_s,\xi_s$ 
satisfy the required smoothness conditions. 
Moreover, these smoothness conditions imply that \eqref{eq:generic-entry} equals
\begin{equation}\label{eq:generic-entry2}
 \frac{ P(\varphi_s,\, \psi_s,\,\xi_s,\, \varphi'_s,\, \psi'_s,\, \xi'_s,\, \varphi''_s,\, \psi''_s, \,\xi''_s)}{\varphi_0^2\,\psi_0^2\,\xi_0^2}+Q(s,r)\,s,
\end{equation}
where $Q$ is continuous. 
Furthermore, by \eqref{eqn:Deltai}, the numerator above
 can be written as a polynomial $\widetilde P$ in the parameter $s$, the functions $\varphi_0,\psi_0,\xi_0$ and their first and second derivatives, and the functions $\Delta_\varphi,\Delta_\psi,\Delta_\xi$ and their first and second derivatives (indicated as $\dots$ below). Thus, \eqref{eq:generic-entry2} and hence \eqref{eq:generic-entry} are equal~to
\begin{equation}\label{eq:generic-entry3}
 \frac{ \widetilde P(s,\, \varphi_0,\, \psi_0,\,\xi_0,\dots, \Delta_\varphi, \, \Delta_\psi,\, \Delta_\xi, \dots)}{\varphi_0^2\,\psi_0^2\,\xi_0^2}+Q(s,r)\,s.
\end{equation}
In particular, the dependence of the above on $s$ is polynomial in the first term, and smooth on the second. Expanding in $s$, we~have
\begin{equation*}
    \widetilde P(s,\, \varphi_0, \psi_0,\xi_0,\dots, \Delta_\varphi, \Delta_\psi, \Delta_\xi, \dots)=\sum_{n=0}^d \widetilde{P}_n(\varphi_0, \psi_0,\xi_0,\dots, \Delta_\varphi, \Delta_\psi,\Delta_\xi, \dots)\,s^n,
\end{equation*}
where $\widetilde{P}_n$ are polynomials. Each coefficient in this sum is a smooth function of $r$ that vanishes at $r=0$ and $r=L$ in such way that the limits of \eqref{eq:generic-entry3} as $r\searrow0$ and $r\nearrow L$ are both finite, 
so the corresponding coefficients in \eqref{eq:generic-entry3}
extend to smooth (hence bounded) functions on $[0,L]$.
Thus, $ \widetilde P(s,\, \varphi_0,\, \psi_0,\,\xi_0,\dots, \Delta_\varphi, \, \Delta_\psi,\, \Delta_\xi, \dots)/\varphi_0^2\,\psi_0^2\,\xi_0^2$ can be regarded as a polynomial in the variable $s$ whose coefficients are \emph{continuous} functions of $r$. We will implicitly (and repeatedly) use this fact in what follows.

\begin{notation*}
    We use $O(s^n)$, respectively $O(r^m)$, to denote any functions of the form $s^n\, F(s,r)$, respectively $r^m\, F(s,r)$, where $F\colon [0,1]\times[0, L]\to\R$ is \emph{bounded}.
\end{notation*}

\subsection{\texorpdfstring{Positive curvature on $S^4$}{Positive curvature on the four-sphere}} 
To simplify the exposition, we shall focus primarily on the case $M^4=S^4$, in which $r_{\mathrm{max}}^\pm=\frac{L}{2}=\frac\pi6$ and it suffices to verify $\sec>0$ along the geodesic segment $\gamma(r)$ with $r\in \left[0, r_{\mathrm{max}}^- \right]$ due to the additional additional symmetries \eqref{eqn:gS4sym}, cf.~\Cref{rem:extra-symm}. 

Let $R_s=\diag\!\big( (R_s)_1,(R_s)_2,(R_s)_3 \big)$ be the curvature operator of $(S^4,\gm_s)$ along $\gamma(r)$, given by \Cref{propn:curv_op}, where $\varphi,\psi,\xi$ are set to be $\varphi_s,\psi_s,\xi_s$ defined in \eqref{eq:def-phis-psis-xis}.
As discussed above, $R_s$, $s\in [0,1]$, extends smoothly to $r=0$, and this extension (as well as its entries) will be denoted by the same symbol(s).
Clearly, $R_0$ is the curvature operator of the Grove--Ziller metric $\gm_0$,
so $R_0+\tau_0\,*\succeq0$ for all $r\in \left[0,r_{\mathrm{max}}^-\right]$, where $\tau_0 := -\frac{\varphi_0'}{2b^2}$, see~\Cref{propn:GZRR}. The proof of \Cref{mainthmB} hinges on the next:

\begin{claim}\label{claim-S4}
If $s>0$ is sufficiently small, then $R_s+\tau_s\,*\succ0$ for all $r\in \left[0,r_{\mathrm{max}}^-\right]$,~with
\begin{equation}\label{eqn:thorpe_s}
 \tau_s(r) := \tau_0(r)+  \frac{2(\sqrt{3}-b)}{b^3}\,s=  -\frac{\varphi_0'(r)}{2b^2} +  \frac{2(\sqrt{3}-b)}{b^3}\,s.
\end{equation}
\end{claim}

We begin the journey towards \Cref{claim-S4} observing that certain diagonal entries of $R_s$, which are sectional curvatures with respect to $\gm_s$, are positive for all $s\in(0,1]$.

\begin{proposition} \label{propn:sec>0_1234}
 For all $s\in(0,1]$ and $r\in \left[0, r_{\mathrm{max}}^-\right]$, the following hold:
\begin{enumerate}[\rm (i)]
    \item $[(R_s)_i]_{22} = \sec_{\gm_s}(e_0\wedge e_i) > 0$ for $1 \leq i \leq 3$;
    \item $[(R_s)_1]_{11} =\sec_{\gm_s}(e_2 \wedge e_3)>0$.
\end{enumerate} 
\end{proposition}

\begin{proof}
As the round metric $\gm_1$ has $\sec \equiv1$, we have $\varphi_1''(r) <0$, $\psi_1''(r) <0$, $\xi_1''(r) <0$ by \Cref{propn:curv_op}, cf.~\eqref{eq:can-phi,psi,xi} and \eqref{eq:can-R-blocks}.
Thus $\varphi_s''(r) <0$, $\psi_s''(r) <0$, $\xi_s''(r) <0$ for all $s\in(0,1]$ and $r\in\left[0,r_{\mathrm{max}}^-\right]$, which implies, by \Cref{propn:curv_op}, that $\sec_{\gm_s}(e_0\wedge e_i) > 0$, for $i=2,3$. 
In the case of $\sec_{\gm_s}(e_0\wedge e_1)$, a further argument is required at $r=0$. Namely, using the smoothness conditions, we see that if $s\in (0,1]$, then
\begin{equation*}
\lim_{r\searrow 0} \sec_{\gm_s}(e_0\wedge e_1)(r) = (1-s)\sec_{\gm_0}(e_0\wedge e_1)(0) + s\sec_{\gm_1}(e_0\wedge e_1)(0) >0,
\end{equation*}
where $(e_0\wedge e_1)(r)$ denotes the $2$-plane in $T_{\gamma(r)}S^4$ spanned by $e_0$ and $e_1$, which concludes the proof of (i).
Regarding (ii), if $s\in (0,1]$ and $r \in \left(0, r_{\mathrm{max}}^-\right]$, then
 \begin{equation*}
 \varphi_s \leq \xi_s < \psi_s, \quad \xi_s' < 0, \quad \psi_s' \geq 0,
 \end{equation*}
 which implies that
  \begin{align*}
\sec_{\gm_s}(e_2\wedge e_3) &= \frac{\psi_s^4+\xi_s^4-\varphi_s^4 + 2(\xi_s^2-\varphi_s^2)(\varphi_s^2-\psi_s^2)}{4\,\varphi_s^2\, \psi_s^2\,\xi_s^2 } - \frac{\psi_s'\xi_s'}{\psi_s\xi_s} \\
&= \frac{(\xi_s^2 - \psi_s^2)^2}{4\,\varphi_s^2\,\psi_s^2 \,\xi_s^2} +\frac{2\psi_s^2 - \varphi_s^2}{4\,\psi_s^2 \,\xi_s^2} +\frac{\xi_s^2-\varphi_s^2}{2\,\psi_s^2 \,\xi_s^2} - \frac{\psi_s'\xi_s'}{\psi_s\xi_s} \geq\frac{b^2}{4\psi_s^2\,\xi_s^2},
  \end{align*}
since $2\psi_s^2-\varphi_s^2\geq \psi_s^2$ and $\psi_s\geq \psi_0\equiv b$ is uniformly bounded from below.
\end{proof}

Let us introduce functions $\eta_i,\mu_i,\nu_i$, $i=1,2,3$, such that the blocks of the curvature operator $\RR_s=\diag\!\big((\RR_s)_1,(\RR_s)_2,(\RR_s)_3\big)$ of $\gm_s$ can be written as a perturbation
 \begin{equation}\label{eq:def-etai-mui-nui}
 \phantom{, \qquad i = 1,2,3,}
  (\RR_s)_i = (\RR_0)_i +  \begin{bmatrix}
    \eta_i(s,r) & \mu_i(s,r)\\
    \mu_i(s,r) & \nu_i(s,r)
  \end{bmatrix}, \qquad i = 1,2,3,
 \end{equation}
 of the blocks of the curvature operator $\RR_0=\diag\!\big((\RR_0)_1,(\RR_0)_2,(\RR_0)_3\big)$ of the Grove--Ziller metric $\gm_0$. Recall that, for $r\in\left(0,r_{\mathrm{max}}^-\right]$, these blocks $(\RR_0)_i$ are computed in \Cref{propn:GZRR}, setting $\varphi=\varphi_0$, i.e., $\varphi$ is given by \eqref{eqn:f1intermsoff}.
Clearly, each of $\eta_i,\mu_i,\nu_i$ is $O(s^n)$ for some $n\geq1$.

\subsubsection{First block}
We first analyze the block $i=1$ of the matrices $R_s$ and $R_s+\tau_s\,*$. 

\begin{proposition}\label{propn:Rs1-new}
 For all $r\in \left[0, r_{\mathrm{max}}^-\right]$, the entries of $(\RR_s)_1$ satisfy:
  \begin{align*}
   \eta_1(s, r) &=  \left(  \frac{3\varphi_0}{2b^5} (\varphi_0 (\Delta_\psi + \Delta_\xi)-b\Delta_\varphi ) - \frac{\Delta_\psi + \Delta_\xi}{b^3} \right)s + O(s^2),\\
   \mu_1(s,r) &= \left( \frac{\varphi_0(\psi_1' + \xi_1')}{2b^3} - \frac{\Delta_\varphi'}{b^2} + \frac{\varphi_0'}{b^3}(\Delta_\psi + \Delta_\xi) \right)s + O(s^2),\\
   \nu_1(s,r) &=  \left( \frac{-\varphi_1''\varphi_0 + \varphi_0''\varphi_1}{\varphi_0^2} \right)s + O(s^2).
  \end{align*}
 \end{proposition}

\begin{proof}
First, let us consider $\eta_1$. From \Cref{propn:curv_op},
\begin{align*}
 [(\RR_s)_1]_{11} &= \frac{\psi_s^4 +\xi_s^4 - \varphi_s^4 + 2(\xi_s^2 - \varphi_s^2)(\varphi_s^2 - \psi_s^2)}{4\varphi_s^2\,\psi_s^2\,\xi_s^2} - \frac{\psi_s'\xi_s'}{\psi_s\xi_s}\\
 &= \frac{(\xi_s^2 - \psi_s^2)^2}{4\varphi_s^2\,\psi_s^2\,\xi_s^2} - \frac{3\varphi_s^2}{4\psi_s^2\,\xi_s^2} + \frac{\xi_s^2 + \psi_s^2}{2\psi_s^2\,\xi_s^2} - \frac{\psi_s'\xi_s'}{\psi_s\xi_s}.
\end{align*}
We analyze these four terms separately using \eqref{eqn:Deltai}, as follows
\begin{align*}
- \frac{3\varphi_s^2}{4\psi_s^2\,\xi_s^2} &= -\frac{3\varphi_0^2}{4b^4} - \frac{3\varphi_0}{2b^5} (b\Delta_\varphi - \varphi_0(\Delta_\psi +\Delta_\xi) )s + O(s^2),\\
 \frac{\xi_s^2 + \psi_s^2}{2\psi_s^2\,\xi_s^2}  &= \frac{1}{b^2} - \frac{\Delta_\psi + \Delta_\xi}{b^3}\,s + O(s^2),\quad  \frac{(\xi_s^2 - \psi_s^2)^2}{4\varphi_s^2\,\psi_s^2\,\xi_s^2} =  O(s^2), \quad - \frac{\psi_s'\xi_s'}{\psi_s\xi_s} = O(s^2).
\end{align*}
Therefore, adding the above together, we find:
\begin{equation*}
 [(\RR_s)_1]_{11} = \frac{4b^2 - 3\varphi_0^2}{4b^4} +\left(  \frac{3\varphi_0}{2b^2} (\varphi_0 (\Delta_\psi + \Delta_\xi)-b \Delta_\varphi ) - \frac{\Delta_\psi + \Delta_\xi}{b^3} \right)s + O(s^2),
\end{equation*}
which establishes the claimed expansion of $\eta_1(s,r)= [(\RR_s)_1]_{11} - \frac{4b^2 - 3\varphi_0^2}{4b^4}$, cf.~\eqref{eq:def-etai-mui-nui}.

Next, consider $\mu_1$. 
From \Cref{propn:curv_op},
\begin{align*}
 [(\RR_s)_1]_{12} &= \frac{\xi_s'(\xi_s^2+\varphi_s^2-\psi_s^2)}{2\varphi_s\,\psi_s\,\xi_s^2} + \frac{\psi_s'(\varphi_s^2+\psi_s^2-\xi_s^2)}{2\varphi_s\,\psi_s^2\, \xi_s} - \frac{\varphi_s'}{\psi_s\,\xi_s}\\ 
 &= \frac{(\xi_s^2 - \psi_s^2)(\xi_s'\psi_s - \psi_s'\xi_s)}{2\varphi_s\,\psi_s^2\,\xi_s^2} + \frac{\varphi_s(\xi_s'\psi_s + \psi_s'\xi_s)}{2\psi_s^2\,\xi_s^2} - \frac{\varphi_s'}{\psi_s\,\xi_s}.
\end{align*}
We analyze these three terms separately, using \eqref{eqn:Deltai}, as before:
\begin{multline*}
 \frac{(\xi_s^2 - \psi_s^2)(\xi_s'\psi_s - \psi_s'\xi_s)}{2\varphi_s\,\psi_s^2\,\xi_s^2} = O(s^2), \quad
\frac{\varphi_s(\xi_s'\psi_s + \psi_s'\xi_s)}{2\psi_s^2\,\xi_s^2} 
 = \frac{\varphi_0(\psi_1' + \xi_1')}{2b^3}\,s + O(s^2),\\
-\frac{\varphi_s'}{\psi_s\,\xi_s} = -\frac{\varphi_0'}{b^2} + \left( \frac{\varphi_0'(\Delta_\psi + \Delta_\xi)}{b^3}-\frac{\Delta_\varphi'}{b^2} \right)s + O(s^2).
\end{multline*}
Thus, adding the above, we have:
\begin{equation*}
 [(\RR_s)_1]_{12} = -\frac{\varphi_0'}{b^2}  + \left( \frac{\varphi_0(\psi_1' + \xi_1')}{2b^3} - \frac{\Delta_\varphi'}{b^2} + \frac{\varphi_0'(\Delta_\psi + \Delta_\xi)}{b^3} \right)s + O(s^2),
\end{equation*} 
which establishes the claimed expansion of $\mu_1(s,r)=[(\RR_s)_1]_{12} +\frac{\varphi_0'}{b^2}$, cf.~\eqref{eq:def-etai-mui-nui}.

Finally, let us consider $\nu_1$. From \Cref{propn:curv_op}, we have:
\begin{equation*}
 [(\RR_s)_1]_{22} = -\frac{\varphi_s''}{\varphi_s}
 = -\frac{\varphi_0''}{\varphi_0} + \left( \frac{-\varphi_1''\varphi_0 + \varphi_0''\varphi_1}{\varphi_0^2} \right)s + O(s^2),
\end{equation*}
which establishes the claimed expansion of $\nu_1(s,r)= [(\RR_s)_1]_{22}+\frac{\varphi_0''}{\varphi_0}$, cf.~\eqref{eq:def-etai-mui-nui}.
\end{proof}

\begin{proposition}\label{prop:Rs1-positive}
If $s>0$ is sufficiently small, then the matrix
\begin{equation*}
  (\RR_s)_1 + \tau_s H = \begin{bmatrix}
    \frac{4b^2 - 3\varphi_0^2}{4b^4} +\eta_1(s,r) & -\frac{3\varphi_0'}{2b^2}+\mu_1(s,r)+\frac{2(\sqrt{3}-b)}{b^3}s \\
  -\frac{3\varphi_0'}{2b^2}+\mu_1(s,r)+\frac{2(\sqrt{3}-b)}{b^3}s & -\frac{\varphi_0''}{\varphi_0}+\nu_1(s,r)
  \end{bmatrix}
 \end{equation*}
 is positive-definite for all $r\in\left[0,r_{\mathrm{max}}^-\right]$.
\end{proposition}

\begin{proof}
The expression above for $(\RR_s)_1 + \tau_s H$ follows from \Cref{propn:GZRR}, as well as  \eqref{eq:defH},  \eqref{eqn:thorpe_s}, and \eqref{eq:def-etai-mui-nui}. 
From \Cref{propn:sec>0_1234} (ii), we know that $[(\RR_s)_1]_{11} > 0$ for all $s\in(0,1]$ and $r\in\left[0,r_{\mathrm{max}}^-\right]$. So, by Sylvester's criterion, it suffices to show that if $s>0$ is sufficiently small, then the following is positive:
 \begin{align*}
   \det\!\big((\RR_s)_1 + \tau_s H\big) &=   \left(\frac{4b^2 - 3\varphi_0^2}{4b^4}\right) \left( -\frac{\varphi_0''}{\varphi_0} \right) - \left( \frac{3\varphi_0'}{2b^2} \right)^2  -\frac{\varphi_0''}{\varphi_0}  \,\eta_1(s,r)\\
   &  \quad +   \frac{4b^2 - 3\varphi_0^2}{4b^4} \, \nu_1(s,r) + \frac{3\varphi_0'}{b^2} \left(\mu_1(s,r)+\frac{2(\sqrt{3}-b)}{b^3}s\right)\\
  & \quad + \eta_1(s,r)\, \nu_1(s,r) - \left(\mu_1(s,r)+\frac{2(\sqrt{3}-b)}{b^3}s\right)^2.
 \end{align*}
By \Cref{propn:Rs1-new}, we have $\det\!\big((\RR_s)_1 + \tau_s H\big)= A(r) + B(r) \, s+ O(s^2)$, where
\begin{align*}
 A(r) &:=  \left(\frac{4b^2 - 3\varphi_0^2}{4b^4}\right) \left( -\frac{\varphi_0''}{\varphi_0} \right) - \left( \frac{3\varphi_0'}{2b^2} \right)^2,\\
B(r) &:=  \left( -\frac{\varphi_0''}{\varphi_0} \right) \left( \frac{3\varphi_0}{2b^5} ( \varphi_0(\Delta_\psi +  \Delta_\xi) -b \Delta_\varphi) - \frac{\Delta_\psi + \Delta_\xi}{b^3}  \right)\\
 &   \quad + \left( \frac{4b^2 - 3\varphi_0^2}{4b^4} \right) \left( \frac{-\varphi_1''\varphi_0 + \varphi_0''\varphi_1}{\varphi_0^2} \right) \\
  & \quad + \frac{3\varphi_0'}{b^2} \, \left( \frac{\varphi_0( \psi_1'+\xi_1')}{2b^3} - \frac{\Delta_\varphi'}{b^2} + \frac{\varphi_0'}{b^3}(\Delta_\psi + \Delta_\xi)  + \frac{2(\sqrt{3}-b)}{b^3}\right).
\end{align*}

Note that $A(r)\geq0$ if $r\in\left[0,r_{\mathrm{max}}^-\right]$ by \Cref{propn:eqnf1_sec>0}, but $A(r) \equiv 0$ near $r = r_{\mathrm{max}}^-$. 
We claim that there exist $0<r_*<r_{\mathrm{max}}^-$ and constants $\alpha>0$ and $\beta>0$ such that 
\begin{equation}\label{eq:claim-a-b}
\begin{aligned}
&A(r)\geq \alpha>0 \mbox{ for all } 0\leq r\leq r_*, \\
&B(r)\geq \beta>0 \mbox{ for all } r_*\leq r\leq r_{\mathrm{max}}^-,
\end{aligned}
\end{equation}
from which it clearly follows that $\det\!\big((\RR_s)_1 + \tau_s H\big) >0$ for all $r\in\left[0,r_{\mathrm{max}}^-\right]$ and sufficiently small $s>0$, as desired. Recall that there exists $0< r_0 < r_{\mathrm{max}}^-$ so that:
 \begin{itemize}
  \item for all $r\in (0,r_0)$, we have $\varphi_0'(r) > 0$ and $\varphi_0''(r)<0$,
  \item for all $r\in \left[r_0,r_{\mathrm{max}}^-\right]$, we have $\varphi_0(r) = b$, and hence $\varphi_0'(r) = \varphi_0''(r) =0$,
 \end{itemize}
cf.~\eqref{eqn:f1intermsoff} and the Grove--Ziller construction (\Cref{subsec:GZ-metrics}). Moreover, for all $\varepsilon >0$, there exists  $0<r_*<r_0$, such that for $r \in \left[r_*, r_{\mathrm{max}}^-\right]$, we have:
 \begin{equation}
 \label{eqn:f1epsilon}
  0 \leq \varphi_0'(r) <\varepsilon, \quad 0 \leq -\varphi_0''(r) < \varepsilon, \;\;\mbox{and}\;\; b - \varepsilon < \varphi_0(r) \leq b, 
 \end{equation}
and these inequalities are strict on $[r_*, r_0)$. 
Thus, choosing $\varepsilon >0$ sufficiently small, we have that for all $r\in \left[r_*, r_{\mathrm{max}}^-\right]$,
 \begin{equation*}
  \frac{-\varphi_1''\varphi_0 + \varphi_0''\varphi_1}{\varphi_0^2}
  %
  = \frac{(2\sin r)(\varphi_0 + \varphi_0'')}{\varphi_0^2}
  \geq  \frac{(2\sin r)(b - 2\varepsilon)}{b^2} >  \frac{1}{4b}.
 \end{equation*}
 Furthermore, by continuity, the following are uniformly bounded on $r\in \left[r_*,r_{\mathrm{max}}^-\right]$,
 \begin{align*}
  \left| -\frac{1}{\varphi_0}  \left(  \frac{3\varphi_0 }{2b^5}(\varphi_0(\Delta_\psi +\Delta_\xi)-b \Delta_\varphi ) - \frac{\Delta_\psi + \Delta_\xi}{b^3} \right) \right| < C_1,\\
  \left| \frac{3}{b^2}  \left( \frac{\varphi_0(\psi_1'+\xi_1')}{2b^3} -\frac{\Delta_\varphi'}{b^2} + \frac{\varphi_0'(\Delta_\psi + \Delta_\xi)}{b^3} + \frac{2(\sqrt{3}-b)}{b^3} \right) \right| < C_2,
 \end{align*}
 where $C_1$ and $C_2$ are constants independent of $r_*$; and  $\left( \frac{4b^2 - 3\varphi_0^2}{4b^4} \right) \geq  \frac{1}{4b^2}$ by \eqref{eqn:f1epsilon}. Putting the above together, and making $\varepsilon>0$ even smaller if needed,  we conclude
 \begin{equation*}
  B(r) > -\varepsilon \, C_1 + \tfrac{1}{16b^3} - \varepsilon \, C_2 = \tfrac{1}{16b^3} - \varepsilon\,(C_1 + C_2) > \beta > 0
 \end{equation*}
 for all $r \in \left[r_*, r_{\mathrm{max}}^- \right]$, where, e.g., $\beta=\tfrac{1}{32b^3}$. Finally, in order to prove the inequality regarding $A(r)$ in \eqref{eq:claim-a-b}, recall there exists $c >0$ such that $\sec_{\gm_{D^2}} \geq c>0$ for all $r\in [0, r_*]$, by \Cref{rem:psec0r1]}. From \eqref{eq:f''intermsofphi}, in the proof of \Cref{propn:eqnf1_sec>0}, we have that
 \begin{equation*}
  \sec_{\gm_{D^2}} = -\frac{f''}{f}  = \frac{ab^2}{(ab^2 - \varphi_0^2)^2}\, \frac{(-\varphi_0'')(ab^2 - \varphi_0^2) - 3\varphi_0\varphi_0'^2}{\varphi_0},
  \end{equation*}
  from which it follows that
\begin{equation*}
\frac{3(ab^2 - \varphi_0^2)^2}{ab^2}\,\sec_{\gm_{D^2}} =
  3\left( -\frac{\varphi_0''}{\varphi_0} \right) (ab^2 - \varphi_0^2) - 9\varphi_0'^2 \leq \left( -\frac{\varphi_0''}{\varphi_0} \right)(4b^2 - 3\varphi_0^2) - 9\varphi_0'^2,
\end{equation*}
because $1<a\leq\frac43$.
Therefore, 
as $\varphi_0(r)<\sqrt{a}\,b$ for all $r$, there exists $\alpha>0$ so that 
\begin{equation*}
A(r) \geq \frac{3}{4} \, \frac{(ab^2 - \varphi_0^2)^2}{ab^2}\, \sec_{\gm_{D^2}} \geq \frac{3}{4} \, \frac{(a b^2 - \varphi_0^2)^2}{ab^2}\, c > \alpha > 0, \;\; \mbox{ for all } r\in [0, r_*].\qedhere
 \end{equation*}
\end{proof}

\subsubsection{Remaining blocks}
We now handle the remaining blocks $i=2,3$.

\begin{proposition}\label{propn:Rs23-new}
 For all $r\in \left[0, r_{\mathrm{max}}^-\right]$, the entries of $(\RR_s)_i$, for $i=2,3$, satisfy: 
  \begin{align*}
   \eta_i(s, r) &= \left( \tfrac{\sqrt{3}}{b} + O(r) \right)s + O(s^2),
   &\mu_i(s,r) &= \left(- \tfrac{2(\sqrt{3}-b)}{b^3} + O(r) \right)s + O(s^2),\\
  & &\nu_i(s,r)&= \left( \tfrac{\sqrt{3}}{b} + O(r) \right)s + O(s^2).
  \end{align*}
 \end{proposition}

\begin{proof}
First, let us consider $\eta_2$. From \Cref{propn:curv_op},
 \begin{equation*}
  [{(\RR_s)_2}]_{11} = \frac{\varphi_s^2}{4 \psi_s^2\,\xi_s^2} + \frac{\psi_s^2 - \xi_s^2}{2 \psi_s^2\,\xi_s^2} + \frac{\xi_s^4 + 2\xi_s^2 \psi_s^2 - 3\psi_s^4 - 4\varphi_s\psi_s^2\xi_s\varphi_s'\xi_s'}{4\varphi_s^2 \,\psi_s^2\,\xi_s^2}.
  \end{equation*}
We analyze these three terms separately using \eqref{eqn:Deltai}. The first two satisfy
\begin{equation*}
  \frac{\varphi_s^2}{4 \psi_s^2\,\xi_s^2 } = \frac{\varphi_0^2}{4b^4} + s\, O(r^2) + O(s^2), \quad \text{ and }\quad   \frac{\psi_s^2 - \xi_s^2}{2 \psi_s^2\,\xi_s^2 }= s\, O(r) + O(s^2),
 \end{equation*}
 while the third satisfies
 \begin{align*}
  \frac{\xi_s^4 + 2 \psi_s^2\xi_s^2 - 3\psi_s^4 - 4\varphi_s\psi_s^2\xi_s\varphi_s'\xi_s'}{4\varphi_s^2\,  \psi_s^2\,\xi_s^2} &=\frac{2(\Delta_\xi - \Delta_\psi) - \varphi_0\varphi_0'\Delta_\xi'}{b \varphi_0^2} \, s + O(s^2)\\
  &= \left(\tfrac{\sqrt{3}}{b} + O(r)\right)\,s + O(s^2),
 \end{align*}
 since $\displaystyle\lim_{r\searrow 0} \tfrac{2(\Delta_\xi - \Delta_\psi) - \varphi_0\varphi_0'\Delta_\xi'}{\varphi_0^2} = \sqrt{3}$,  by L'H\^{o}pital's rule and \Cref{prop:smoothness} (i).
 
 Altogether, the above yields $ [(\RR_s)_2]_{11} = \frac{\varphi_0^2}{4b^4} + \left( \frac{\sqrt{3}}{b} + O(r) \right)s + O(s^2)$, and hence establishes the claimed expansion of $\eta_2(s,r) = [(\RR_s)_2]_{11} -\tfrac{\varphi_0^2}{4b^4}$, cf.~\eqref{eq:def-etai-mui-nui}.

Second, the proof that $\eta_3$ has the same expansion as $\eta_2$ is similar. Namely, 
\begin{align*}
  [(\RR_s)_3]_{11} &= \frac{\varphi_s^2}{4\psi_s^2\,\xi_s^2 } + \frac{(\psi_s^2 - \xi_s^2)(\psi_s^2 + 3\xi_s^2 - 2\varphi_s^2) - 4\varphi_s\psi_s\xi_s^2\varphi_s'\psi_s'}{4\varphi_s^2\, \psi_s^2\, \xi_s^2},
 \end{align*}
 where the first term was already considered above, and the second term satisfies
 \begin{equation*}
  \frac{(\psi_s^2 - \xi_s^2)(\psi_s^2 + 3\xi_s^2 - 2\varphi_s^2) - 4\varphi_s\psi_s\xi_s^2\varphi_s'\psi_s'}{4\varphi_s^2\,\psi_s^2 \,\xi_s^2 }\\
  =\left( \frac{\sqrt{3}}{b} + O(r) \right)s + O(s^2),
 \end{equation*}
by similar considerations involving  L'H\^{o}pital's rule and \Cref{prop:smoothness} (i).
Thus, $\eta_3(s,r) = [(\RR_s)_3]_{11}  - \tfrac{\varphi_0^2}{4b^4} = \left( \frac{\sqrt{3}}{b} + O(r) \right)s + O(s^2)$, cf.~\eqref{eq:def-etai-mui-nui}.

Next, consider $\mu_2$. From \Cref{propn:curv_op},
 \begin{equation*}
  [(\RR_s)_2]_{12}  = \frac{\varphi_s'}{2\psi_s\,\xi_s} + \frac{\varphi_s'\xi_s(\psi_s^2 - \xi_s^2) + \varphi_s\xi_s'(\xi_s^2+\psi_s^2-\varphi_s^2) - 2\varphi_s\psi_s\xi_s\psi_s'}{2\varphi_s^2\,\psi_s\,\xi_s^2}
  \end{equation*}
The first term above satisfies
\begin{equation*}
  \frac{\varphi_s'}{2\xi_s\psi_s} = \frac{\varphi_0'}{2b^2} +\left( O(r^2) - \frac{2(\sqrt{3}-b)}{b^3} \right)s + O(s^2),
  \end{equation*}
  while the second satisfies
\begin{equation*}
\frac{\varphi_s'\xi_s(\psi_s^2 - \xi_s^2) + \varphi_s\xi_s'(\xi_s^2+\psi_s^2-\varphi_s^2) - 2\varphi_s\psi_s\xi_s\psi_s'}{2\varphi_s^2\,\psi_s\,\xi_s^2}
= s\, O(r) + O(s^2).
\end{equation*}
So, $\mu_2(s,r) = [(\RR_s)_2]_{12} -\tfrac{\varphi_0'}{2b^2}=\left(- \frac{2(\sqrt{3}-b)}{b^3} + O(r) \right)s + O(s^2)$, cf.~\eqref{eq:def-etai-mui-nui}.
The proof that $\mu_3$ has the same expansion as $\mu_2$ is analogous, and left to the reader.

Finally, let us consider $\nu_2$ and $\nu_3$. From \Cref{propn:curv_op} and \eqref{eq:def-etai-mui-nui}, we have
\begin{equation*}
     \nu_2(s,r) =[(\RR_s)_2]_{22} =-\tfrac{\psi_s''}{\psi_s}\quad\text{ and } \quad \nu_3(s,r) =[(\RR_s)_3]_{22} = -\tfrac{\xi_s''}{\xi_s}.
\end{equation*}
By \eqref{eqn:Deltai}, we have $\psi_s''  = \Delta_\psi''  \, s= \psi_1'' \,s$ and $\xi_s''  = \Delta_\xi'' \,s  = \xi_1'' \,s$, so
\begin{equation*}
      \nu_2(s,r) = \left( \tfrac{\sqrt{3}}{b} + O(r) \right)s + O(s^2), \; \text{ and }\; \nu_3(s,r) = \left( \tfrac{\sqrt{3}}{b} + O(r) \right)s + O(s^2).\qedhere
\end{equation*}
\end{proof}

\begin{proposition}\label{prop:Rs23-positive}
If $s>0$ is sufficiently small, then the matrices
\begin{equation}\label{eq:modified-blocks-23}
  (\RR_s)_i + \tau_s H = \begin{bmatrix}
    \frac{\varphi_0^2}{4b^4} + \eta_i(s,r) & \mu_i(s,r)+\frac{2(\sqrt3-b)}{b^3}s\\
    \mu_i(s,r)+\frac{2(\sqrt3-b)}{b^3}s & \nu_i(s,r)
  \end{bmatrix}, \quad i=2,3,
 \end{equation}
 are positive-definite for all $r\in\left[0, r_{\mathrm{max}}^-\right]$.
\end{proposition}
 
\begin{proof}
The expression \eqref{eq:modified-blocks-23} for $(\RR_s)_i + \tau_s H$, $i=2,3$, follows from \Cref{propn:GZRR}, as well as \eqref{eq:defH}, \eqref{eqn:thorpe_s}, and \eqref{eq:def-etai-mui-nui}. First, consider the $(1,1)$-entry of these matrices:
  \begin{equation*}
   \phantom{\quad \text{ for }i=2,3.}[(\RR_s)_i]_{11} = \tfrac{\varphi_0^2}{4b^4} + \left( \tfrac{\sqrt{3}}{b} + O(r) \right)s + O(s^2), \quad \text{ for }i=2,3,
  \end{equation*}
cf.~\Cref{propn:Rs23-new}. Since $\varphi_0(r)>0$ away from $r=0$, and the $O(s)$ part of the above is uniformly positive near $r=0$, it follows that $[(\RR_s)_i]_{11}>0$ for all $r\in\left[0, r_{\mathrm{max}}^-\right]$ and $i=2,3$, provided $s>0$ is sufficiently small.

Second, let us analyze the determinant of \eqref{eq:modified-blocks-23}. By \Cref{propn:Rs23-new},
\begin{align*}
  \eta_i(s,r)  \nu_i(s,r) &= \left( \tfrac{3}{b^2} + O(r) \right)s^2 + O(s^3),\\
   \mu_i(s,r)+\tfrac{2(\sqrt3-b)}{b^3}s &= 
  s\, O(r) + O(s^2).
 \end{align*}
Thus, using that $\nu_i(s,r)=[(R_s)_i]_{22}$, for $i=2,3$, we have:
 \begin{align*}
  \det\!\big( (\RR_s)_i + \tau_s H\big) &= \nu_i(s,r)\, \tfrac{\varphi_0^2}{4b^4} + \left( \tfrac{3}{b^2} + O(r) \right)s^2  + O(s^3)\\
  &=  [(R_s)_i]_{22} \, \tfrac{\varphi_0^2}{4b^4} + \left( \tfrac{3}{b^2} + O(r) \right)s^2 + O(s^3).
 \end{align*}
By \Cref{propn:sec>0_1234} (i), the $O(s)$ part of the above is positive for $r\in\left(0,r_{\mathrm{max}}^- \right]$, but vanishes at $r= 0$, as $\varphi_0(0)=0$. Since the $O(s^2)$ part has a positive limit as $r\searrow 0$, we have that $\det\!\big( (\RR_s)_i+ \tau_s H\big) >0$ for all $r\in \left[0, r_{\mathrm{max}}^-\right]$ and $i=2,3$, if $s>0$ is sufficiently small. Positive-definiteness now follows from Sylvester's criterion.
 \end{proof}

The above \Cref{prop:Rs1-positive,prop:Rs23-positive} imply \Cref{claim-S4}, since $R_s+\tau_s\,*$ is block diagonal with blocks $(R_s)_i+\tau_s H$, $i=1,2,3$, see \Cref{propn:curv_op} and \eqref{eq:defH}. In turn, \Cref{claim-S4} and the Finsler--Thorpe trick (\Cref{prop:FTtrick}) imply that $\sec_{\gm_s}>0$ for sufficiently small $s>0$. This proves \Cref{mainthmB} for $M^4=S^4$; since, if the original Grove--Ziller metric $\gGZ$ was rescaled as $\gm_0=\lambda^2\,\gGZ$ to standardize $L=\frac\pi3$, then $\lambda^{-2}\,\gm_s$ has $\sec>0$ and is arbitrarily $C^\infty$-close to $\gGZ$ for $s>0$ sufficiently small.

\subsection{\texorpdfstring{Positive curvature on $\C P^2$}{Positive curvature on the complex projective plane}}
We now briefly discuss the proof of \Cref{mainthmB} for $M^4=\C P^2$. Recall that, in this case, $L=\tfrac{\pi}{4}$, with $r_{\mathrm{max}}^-=\frac\pi6$ and $r_{\mathrm{max}}^+=\frac{\pi}{12}$.
Differently from $S^4$, for $M^4=\C P^2$, the situation on the intervals $[0,r_{\mathrm{max}}^-]=\left[0, \tfrac{\pi}{6}\right]$ and $[r_{\mathrm{max}}^-,L]=\left[\tfrac{\pi}{6},\tfrac{\pi}{4}\right]$ has to be analyzed separately, cf.~\Cref{rem:extra-symm}. 

Denoting by $R_0$ the curvature operator of the Grove--Ziller metric $\gm_0$ on $\C P^2$, the function $\tau_0\colon [0,L]\to\R$ so that $R_0+\tau_0\,*\succeq0$ for all $r\in [0,L]$ is given by
\begin{equation*}
  \tau_0(r) = \begin{cases}
      -\frac{\varphi_0'(r)}{2b^2}, & \text{ if } r \in \left[0,r_{\mathrm{max}}^-\right], \\[5pt]
      -\frac{\xi_0'(r)}{2b^2}, & \text{ if } r \in \left[r_{\mathrm{max}}^-,L\right],
  \end{cases}
 \end{equation*}
cf.~\Cref{propn:GZRR}. Note that $\varphi_0'=\xi_0'=0$ near $r=r_{\mathrm{max}}^-$. The proof of \Cref{mainthmB} follows in the same way as in the case $M^4=S^4$ above, replacing \Cref{claim-S4} with:

\begin{claim}\label{claim-CP2}
If $s>0$ is sufficiently small, then $R_s+\tau_s\,*\succ0$ for all $r\in \left[0,L\right]$,~where    
\begin{equation*}
  \tau_s(r) := 
   \begin{cases}
      -\frac{\varphi_0'(r)}{2b^2} + \left(\frac{3}{2b} +\frac{1-b}{b^3}\right)s, & \text{ if } r \in \left[0,r_{\mathrm{max}}^-\right], \\[5pt]
      -\frac{\xi_0'(r)}{2b^2} +  \tfrac{\sqrt{2}\, -\, 2b}{b^3} \, s, & \text{ if } r \in  \left(r_{\mathrm{max}}^-,L\right].
  \end{cases}
 \end{equation*}
\end{claim}

\begin{remark}
Similarly to \eqref{eqn:thorpe_s} in \Cref{claim-S4}, the above function $\tau_s$ is obtained from $\tau_0$ by adding a locally constant multiple of $s$. This $O(s)$ perturbation is not constant as in the case of $M^4=S^4$, and, as a result, $\tau_s(r)$ is \emph{discontinuous} at $r=r_{\mathrm{max}}^-$ for all $s>0$. Nevertheless, the application of the Finsler--Thorpe trick (\Cref{prop:FTtrick}) is pointwise and no regularity is needed.
A posteriori, a continuous function $\widetilde\tau_s(r)$ such that $R_s+\widetilde{\tau_s}\,*\succ0$ for all sufficiently small $s>0$ can be chosen, e.g., as the midpoint $\widetilde\tau_s(r)=\frac{1}{2}(\tau_{\mathrm{min}} +\tau_{\mathrm{max}})$ of $[\tau_{\mathrm{min}},\tau_{\mathrm{max}}]$ for each $r\in [0,L]$, see \Cref{rem:set-of-taus}.
\end{remark}

The proof of \Cref{claim-CP2} follows the same template from \Cref{claim-S4}, relying on expansions in $s$ of the functions $\eta_i,\mu_i,\nu_i$, cf.~\eqref{eq:def-etai-mui-nui}.
The statement of \Cref{propn:Rs1-new}, regarding $i=1$ and $r\in [0, r_{\mathrm{max}}^-]$, holds \emph{tout court} for $\C P^2$, since the smoothness conditions of $\varphi,\psi,\xi$ at $r=0$ are not used in the proof. The case of $i=3$ and $r\in [r_{\mathrm{max}}^-,L]$ is analogous.
The replacement for \Cref{propn:Rs23-new} is the following: 

\begin{proposition}\label{propn:RsCP2}
For $r\in \left[0, r_{\mathrm{max}}^-\right]$, the entries of $(R_s)_i$, $i=2,3$, satisfy: 
 \begin{align*}
  \eta_2(s,r) &= \left(\tfrac{1}{b} + O(r)\right)s + O(s^2), &  \mu_2(s,r) &= \left(-\tfrac{3}{2b} -\tfrac{1-b}{b^3} + O(r)\right)s + O(s^2),\\
  & &\nu_3(s,r) &= \left(\tfrac{4}{b} + O(r)\right)s + O(s^2),\\
  \eta_3(s,r) &= \left(\tfrac{4}{b} + O(r)\right)s + O(s^2), & \mu_3(s,r) &= \left( \tfrac{3}{2b} -\tfrac{1-b}{b^3}+ O(r)\right)s + O(s^2),\\
  &  &\nu_2(s,r) &=\left(\tfrac{1}{b} + O(r)\right)s  + O(s^2).
 \end{align*}
For $r\in \left[r_{\mathrm{max}}^-,L\right]$,  setting $z=L-r$, the entries of $(R_s)_i$, $i=1,2$, satisfy:
 \begin{align*}
\eta_i(s,z) &= \left(\tfrac{1}{b\sqrt{2}} + O(z)\right)s + O(s^2), &  \mu_i(s,z) &= \left( -\tfrac{\sqrt{2} - 2b}{b^3} + O(z)\right)s + O(s^2),  \\
 & &\nu_i(s,z) &= \left(\tfrac{1}{b\sqrt{2}} + O(z)\right)s + O(s^2).
 \end{align*}
\end{proposition}

The proof of \Cref{propn:RsCP2} is totally analogous to that of \Cref{propn:Rs23-new}; noting that, in terms of $z = L - r\in\left[0,r_{\mathrm{max}}^+\right]$, the functions $\varphi_1,\psi_1,\xi_1$ are:
 \begin{equation*}
  \textstyle \varphi_1(z) = \frac{1}{\sqrt{2}}\left(\cos z - \sin z\right), \quad \psi_1(z) = \frac{1}{\sqrt{2}}\left(\cos z + \sin z\right), \quad \xi_1(z) = \sin 2z.
 \end{equation*}
Finally, similarly to \Cref{prop:Rs1-positive,prop:Rs23-positive}, it can be shown that $(R_s)_i+\tau_sH$, $i=1,2,3$, are positive-definite for all $r\in [0,L]$ and $s>0$ sufficiently small, which proves \Cref{claim-CP2} (and hence \Cref{mainthmB}) for $\C P^2$. Details are left to the~reader.

\section{Positive turns negative}
\label{sec:pos-neg}

In this section, we prove \Cref{mainthmA}, using the fact that Grove--Ziller metrics on $S^4$ and $\C P^2$ immediately acquire negatively curved planes under Ricci flow~\cite{bettiol-krishnan1}, together with \Cref{mainthmB}, and continuous dependence on initial data~\cite{BGI20}.

\begin{proof}[Proof of \Cref{mainthmA}]
Let $M^4$ be either $S^4$ or $\C P^2$, and consider the $1$-parameter family of metrics $\gm_s$ on $M^4$, defined in \eqref{eq:gs}, such that $\gm_0$ is a Grove--Ziller metric and $\gm_1$ is either the round metric or the Fubini--Study metric, accordingly. From \Cref{propn:gssmooth}, the metrics $\gm_s$ are  smooth, and it is evident from \eqref{eq:def-phis-psis-xis} and \eqref{eq:gs} that, for all $k\geq0$ and $0<\alpha<1$, there exists a constant $\lambda_{k,\alpha}>0$ such that
\begin{equation}\label{eq:gs-g0}
    \phantom{\quad \text{ for all } 0\leq s\leq 1.}\|\gm_s-\gm_0\|_{C^{k,\alpha}}\leq \lambda_{k,\alpha}\, s, \quad \text{ for all } 0\leq s\leq 1,
\end{equation}
where $\|\cdot\|_{C^{k,\alpha}}$ denotes the H\"older norm on sections of the bundle $E=\Sym^2 TM^4$ with respect to a fixed background metric.
For $0\leq s\leq 1$, let $\gm_s(t)$, $0\leq t < T(\gm_s)$, be the maximal solution to Ricci flow starting at $\gm_s(0)=\gm_s$, where $0<T(\gm_s)\leq +\infty$ denotes the maximal (smooth) existence time of the flow. For all $0\leq s\leq 1$ and $0\leq t<T(\gm_s)$, we have that $\gm_s(t)\in C^\infty(E)$, so $\gm_s(t)$ is in the proper closed subspace $h^{k,\alpha}(E) \subset C^{k,\alpha}(E)$ for all $k\geq0$ and $0<\alpha<1$, in the notation of \cite{BGI20}. 

From the main theorem in \cite{bettiol-krishnan1}, there exist a $2$-plane $\sigma$ tangent to $M^4$ and $t_0>0$ such that $\sec_{\gm_0}(\sigma)=0$ and $\sec_{\gm_0(t)}(\sigma)<0$ for all $0<t<t_0$. Fix $0<t_*<t_0$, and let $\delta>0$ be such that $\sec_\gm(\sigma)<0$ for all metrics $\gm$ with $\|\gm-\gm_0(t_*)\|_{C^{2,\alpha}}<\delta$.
By the continuous dependence of Ricci flow on initial data~\cite[Thm A]{BGI20}, there exist constants $\mathrm r>0$ and $C>0$, depending only on $t_*$ and $\gm_0$, such that, if $\|\gm_s-\gm_0\|_{C^{4,\alpha}}\leq \mathrm r$, then $T(\gm_s)\geq t_0$ and $\|\gm_s(t) - \gm_0(t)\|_{C^{2,\alpha}} \leq C \|\gm_s - \gm_0\|_{C^{4,\alpha}}$ for all $t\in [0, t_0]$.
Together with \eqref{eq:gs-g0}, this yields that if $0\leq s\leq \mathrm r/\lambda_{4,\alpha}$, then
\begin{equation*}
 \|\gm_s(t) - \gm_0(t)\|_{C^{2,\alpha}}
 \leq C\,\|\gm_s - \gm_0\|_{C^{4,\alpha}} 
 \leq C\, \lambda_{4,\alpha}\, s, \quad \text{ for all } 0\leq t\leq t_0.
\end{equation*}
Thus, $\|\gm_s(t_*)-\gm_0(t_*)\|_{C^{2,\alpha}}<\delta$ and so $\sec_{\gm_s(t_*)}(\sigma)<0$, for all $0\leq s<\delta/(C \,\lambda_{4,\alpha})$,
while $\gm_s=\gm_s(0)$ has $\sec>0$ if $s>0$ is sufficiently small, by \Cref{mainthmB}.
\end{proof}

\begin{remark}\label{rem:max-princ}
The curvature operators $R(t)\colon \wedge^2 TM\to\wedge^2 TM$ of metrics $\gm(t)$ on $M^n$ evolving under Ricci flow satisfy the PDE $\frac{\partial}{\partial t}R=\Delta R+2Q(R)$, where $Q(R)$ depends quadratically on $R$. By Hamilton's Maximum Principle, if an $\mathsf O(n)$-invariant cone $\mathcal C\subset \Sym^2_{\mathrm b}(\wedge^2TM)$ is preserved by the ODE $\frac{\dd}{\dd t}R=2Q(R)$, then it is also preserved by the above PDE. 
It was previously known that the cone $\mathcal C_{\sec>0}$ of curvature operators with $\sec>0$ is \emph{not} preserved under the above ODE on $R$ in dimensions $n\geq4$, since it is easy to find $R_0\in\partial \mathcal C_{\sec>0}$ with $Q(R_0)$ pointing outside of $\mathcal C_{\sec>0}$. Nevertheless, this observation alone does not imply the existence of metrics realizing such a family of curvature operators on some closed $n$-manifold, thus evolving under Ricci flow and losing $\sec>0$, as the above metrics $\gm_s(t)$ do.
\end{remark}

\end{document}